\documentclass[
]
{article}
\usepackage[dvips]{geometry}

\usepackage[numbers]{natbib}
\usepackage[center]{titlesec}

\usepackage{amsmath}
\usepackage{amsfonts}
\usepackage{amsthm}

\usepackage{bbold}
\usepackage{fancybox}
\usepackage{tikz}

\usepackage{mathdots}

\usepackage{setspace}

\usepackage{titlesec}
\titleformat{\subsection}[runin]
{\normalfont\large\bfseries}{\thesubsection}{1em}{}

\DeclareMathOperator{\rg}{rg}
\DeclareMathOperator{\Res}{Res}
\DeclareMathOperator{\id}{id}
\DeclareMathOperator{\Jac}{Jac}

\DeclareMathOperator{\PGCD }{PGCD }
\DeclareMathOperator{\im}{Im}

\DeclareMathOperator{\rk}{rk}

\usepackage{amssymb}
\usepackage[all]{xy}
\usepackage[nottoc, notlof, notlot]{tocbibind}
\usepackage{natbib}
\usepackage{pb-diagram}
\usepackage{pb-xy}
\usepackage{amssymb} 
\usepackage{graphicx} 
\usepackage{blkarray}
\usepackage{titling, lipsum}

\usepackage{tikz-cd}

\newtheorem{theo}{Théorème }[section]
\newtheorem{lem}{Lemme}[section]
\newtheorem{prop}{Proposition}[section]
\newtheorem{coro}{Corollaire}[section]
\newtheorem{rem}{Remarque}[section]
\newtheorem{defi}{Définition}[section]

\newtheorem*{prop*}{Proposition}
\newtheorem*{defi*}{Définition}
\newtheorem*{rem*}{Remarque}
\usepackage{cancel}
\makeatletter
\newcommand*{\encircled}[1]{\relax\ifmmode\mathpalette\@encircled@math{#1}\else\@encircled{#1}\fi}
\newcommand*{\@encircled@math}[2]{\@encircled{$\m@th#1#2$}}
\newcommand*{\@encircled}[1]{%
 \tikz[baseline,anchor=base]{\node[draw,circle,outer sep=0pt,inner sep=.2ex] {#1};}}

\title{Stratifications des fibres singulières des systèmes de Mumford}
\author{ {Yasmine Fittouhi} \\
\\
\\
Department of Mathematics\\
Faculty of Natural Sciences\\
University of Haifa\\
199 Abba Khoushy Avenue\\
Mount Carmel\\
Haifa, 3498838, Israel }

\date{\today}
\begin{document}
\maketitle
Mots clés : Système intégrable, stratification, singularité, courbe hyperelliptique.\\
\bigskip

\textbf{Résumé:} Un système intégrable est un système dynamique 
 caractérisé par 
l'existence de constantes de mouvement
et l'existence d'invariants algébriques, ayant une base en géométrie algébrique.\\ 
Dans les années 1970, Mumford introduit un nouveau système complètement intégrable défini sur une courbe hyperelliptique lisse. Dans les années 2000, Vanhaecke a complété la description du système intégrable de Munford en définissant une structure de Poisson sur l'espace de phase du système de Mumford.

Dans cet article nous étudierons le système de Mumford singulier. Le point de depart consiste à déterminer quand et pourquoi le système de Mumford est singulier. Pour cela on fera une étude approfondie pour comprendre ce qui se passe aux singularités, en utilisant le concept de stratification. Nous définirons deux stratifications de l'espace de phase, une stratification algébrique et l'autre stratification géométrique. On prouvera que ces stratifications sont identiques et elles nous permettront de définir une stratification plus fine sur chaque fibre du système de Mumford. Nous conclurons cet article par le résultat étonnant suivant: chaque strate d'une fibre est une partition de sous variétés quasi-affine équidimentionelles. 

\bigskip

\textbf{Remerciements:} Je tiens à remercier Professeur Pol Vanhaecke qui m’a fait découvrir les systèmes intégrables et leurs connexions avec les autres disciplines mathématiques, et de m’avoir dirigée et guidée
vers les systèmes de Mumford.
\ 

Plus que tout je suis reconnaissante au Professeur Antony Joseph qui m'a aidée à mettre en valeur et à structurer mes résultats.

\section{introduction} 

%

\subsection*{1} La mécanique classique moderne est décrite à travers la mécanique Hamiltonienne 
où les coordonnées 
positions et vitesses des mobiles sont regroupées pour former un ensemble qu'on appelle l'espace de phases, 
 cet espace admet 
une structure de Poisson. 
Jacobi a reformulé 
la mécanique Hamiltonienne
en utilisant le puissant formalisme du crochet de Poisson
où l'évolution temporelle des variables canoniques $q$ 
est donnée par un
Hamiltonien $H$ 
via l'équation ${\partial q}/{\partial t} =\{H, q\}$.\\

Un système hamiltonien a des chances d'être résolu s'il a suffisamment de constantes de mouvement. Rappelons qu'une constante de mouvement est une $F$ fonction 
de l'espace de phase 
 indépendante du temps 
telle que $\partial F/\partial t = \{F, H\} = 0$
. Le crochet de Poisson dévoile sa force lors de la recherche des constantes de mouvement 
car leurs crochets de Poisson commutent avec l'Hamiltonien. 
\\
En effet 
un système dynamique conserve l'énergie car son Hamiltonien est indépendant du temps, $\partial H/\partial t = \{H, H\} = 0$. 

\medskip

 Un système dynamique est dit complètement integrable au sens de Liouville s'il satisfait les deux propriétés suivantes:\\
$\bullet $ La première propriété est que l'espace de phase $M$ soit une variété de Poisson lisse de dimension $2n$ munie du crochet de Poisson $\{\cdot, \cdot\}$ de rang $2n$.\\ 
$\bullet $ La seconde propriété est la donnée de $n$ fonctions lisses $(F_i)_{1 \leqslant i \leqslant n} : M \rightarrow \mathbb{C}^n$ qui génèrent $n$ champs vectoriels $(\chi_{F_i}=\{., F_i\})_{1 \leqslant i \leqslant n}$ 
 linéairement indépendants sur un ouvert dense de $M$ 
 , 
 avec $\{F_i, F_j\}=0 $ pour tout ${1 \leqslant i, j \leqslant n}$. \\
Un système Hamiltonien complètement intégrable est dit maximal 
lorsque les champs de vecteurs $(\chi_{F_i})_{1 \leqslant i \leqslant n}$ sont linéairement indépendants en tout point de l'espace de phase.

\

Un point $ m\in M$ est dit régulier si les champs de vecteurs associés à la famille de fonctions $F=(F_i)_{1 \leqslant i \leqslant n}$ par la structure de Poisson sont linéairement indépendants au point $m$. Un point $c\in \mathbb{C}^n$ est une valeur régulière si la fibre $F^{-1}(c)$ ne contient que des points réguliers.\\
Soit $c$ une valeur régulière de $ \mathbb{C}^n$, le théorème de Arnold-Liouville (voir \cite[ page 342 , théorème 12.11]{LGV}) affirme que la fibre $ F^{-1}(c)$ est difféomorphe à $\mathbb{C}^{2n-k}\times \mathbb{T}^{k}$ où $ \mathbb{T}$ est un tore. On appelle la partie torique de $ F^{-1}(c)$ le tore de Liouville. 
\\
Les points singuliers de $M$, (respectivement les valeurs singulières de $\mathbb{C}^n$) sont les points (respectivement les valeurs) qui ne vérifient pas les conditions de régularité.\\

Dans la section \ref{section2}, les fibres du système décrites par Mumford et Vanhaecke (voir théorème \ref{in}) sont au dessus des points réguliers.\\
Notamment, dans la section \ref{section3} les fibres du système que nous décrivions sont au dessus des valeurs singulières (voir proposition \ref{SM}).
\ 
\medskip
\

\subsection*{2.} 
Les systèmes hamiltoniens complètement intégrables admettent une solution, néanmoins nous avons toujours une expression explicite de leurs solutions. 
Les deux exemples suivants: le système de Kepler et le système de Toda, sont des systèmes complètement intégrables maximaux et on connait leur solution explicitement.\\
%
L'espace de phase du système de Kepler $\mathbf{T}^*\mathbb{R}^3 =\langle q_1, q_2, q_3, p_1, p_2, p_3\rangle$ ou $q=(q_i)_{1 \leqslant i \leqslant 3}$ est le vecteur position, les $p=(p_i)_{1 \leqslant i \leqslant 3}$ est le vecteur moment d'une planète. L'espace $\mathbf{T}^*\mathbb{R}^3 $ est muni d'un crochet de Poisson canonique\footnote{où les seuls crochets non nuls sont $\{q_i , p_i\}=1$ pour $i\in\{1, 2, 3\}$.}.
L'hamiltonien du systeme de Kepler est $H =1/2 \sum\limits^3_{i=1}p_i^2+V(r)$
où $V(r)=1/r$ 
avec $r 
$ la distance entre les deux corps (planète et soleil). \\
Le théorème de Noether nous affirme que le moment cinétique $L = q\times p$ est une constante du mouvement pour toute charge centrale. 
Le système de Kepler admet une
 constante de mouvement supplémentaire appelée vecteur de Laplace-Runge-Lenz (vecteur LRL) et est notée
$A = L \times p + \frac{q}{r}$
avec $\{H, A\} =0$.
La conservation du vecteur de LRL est associée à une symétrie 
cachée mise en évidence par la représentation mathématique qui utilise l'inverse de la projection stéréographique de l'espace de phase du problème de Kepler et l'identifie
 à une particule se déplaçant librement sur une sphère; 
ainsi on prouve que le système de Kepler est complètement résoluble d'une manière époustouflante.
\

En 1967, le
physicien Toda a défini un système intégrable multidimensionnel qui porte son nom.
Le système de Toda est un système integrable. On trouve les details de ce système dans \cite{B} et \cite{BK}.


\bigskip
 
\subsection*{3.} Dans les années 70, Mumford tombe sur un nouveau système complètement intégrable, lors 
de son étude du diviseur thêta de la jacobienne d'une courbe hyperelliptique lisse 
 de genre $g$. 
 On sait que la jacobienne d'une courbe hyperelliptique lisse est un tore. Dans cette optique, Mumford a développé un système complètement integrable dont les fibres régulières sont isomorphes aux tores de Liouville. La similitude entre une fibre régulière d'un système complètement integrable 
 et la jacobienne d'une courbe est qu'ils soient tous deux isomorphes à un tore.

Soit $\mathcal{C}: y^2=h(x)$ une courbe hyperelliptique lisse de genre $g$, Mumford 
 a réussi à établir le lien entre la jacobienne $\Jac(\mathcal{C})$ de $\mathcal{C}$ et la fibre d'une valeur régulière du système integrable. La méthode utilisée par Mumford pour passer d'une courbe hyperelliptique $\mathcal{C}$ à un système intégrable, 
 est fort intéressante; Cependant, Mumford explique peu les motivations qui l'ont mené à développer le système integrable (voir section \ref{section2} partie \ref{Meth1}). \\
Soient $g$ points génériques distincts $(x_i, y_i)$ sur $\mathcal{C}$. On définit trois polynômes, le premier polynôme $u(x)=\prod_i (x-x_i)$, le deuxième polynôme $v$ doit satisfaire cette condition $v(x_i)=y_i$, Le troisième polynôme $w(x)$ est obtenu par construction en divisant $h(x)-v^2(x)$ par $u(x)$ (voir section \ref{section2} équation (\ref{uvx''}) pour la définition formelle de $w$.).\\
Notez que $u(x)w(x) +v^2(x) =h(x)$, (le polynôme $h$ est l'hamiltonien qu'on considérera). 
Le polynôme $h$ coïncide à moins $(-)$ le déterminant de la matrice de trace-zéro, dont les entrées 
sont les polynômes $u, v$ et $w$, 
$h=-\det 
\left(\begin{array}{cc}
v &{u} \\
w & -v
\end{array} \right)$. \\
L'ensemble des $g$ points génériques de la courbe $\mathcal{C}$ forment le groupe appelé la jacobienne de la courbe et est noté $\Jac(\mathcal{C})$.\\
Soient $a, b, c \in \Jac(\mathcal{C})$ et soit $\ell_c$ la translation à gauche par $c$ sur $\Jac(\mathcal{C}) $ tels que $b=\ell_c(a)=ca$. Les espaces tangents aux points $a$ et $b$ sont isomorphes par la différentiation de la translation à gauche $D(\ell_{c}) : \mathbb{T}_a{\Jac(\mathcal{C})} \rightarrow \mathbb{T}_{b}{\Jac(\mathcal{C})} .$
On note par $1$ l'élément identité du groupe $ \Jac(\mathcal{C})$, soit $\chi$ un champ de vecteur. On dit qu'un champ de vecteurs est invariant par translation si $\chi_c= D(\ell_{c}) [\chi_1 ]$ pour tout $c\in {\Jac(\mathcal{C})}$. Rappelons que les champs vectoriels de $\Jac(\mathcal{C})$ forment une algèbre de Lie.\\
Mumford a obtenu un champ de vecteurs invariant par translation pour définir un champs hamiltonien, ce champ de vecteurs est le point de départ pour définir l'évolution des $g$ points aux travers du flot de ce dernier, ce qui lui permet de définir 
 l'évolution temporelle de $u, v, w$ par (\ref{m2}) et (\ref{m11}) et il déduit que l'évolution temporelle de 
 $h$ est nulle (e.i $\partial h/\partial t=0$)! cette contexture est un miracle supplémentaire (voir section \ref{section2}). L'égalité $\partial h/\partial t=0 $ implique que 
les coefficients des puissances de $x$ du polynôme $h(x)$ sont des constantes de mouvement c'est-à-dire invariants par l'évolution temporelle. 
Ainsi Mumford a défini un système dynamique sur l'espace de phase composé de $g$ points génériques de la courbe $\mathcal{C}$.

\

Comme l'a souligné Hitchin \cite{lettre}, on peut se demander s'il y a d'autres variétés 
conduisant à des systèmes intégrables complets ?

\medskip

\subsection*{4.} La section \ref{section2} de cet article est un exposé 
 détaillé de la procédure établie par Vanhaecke pour définir le système de Mumford ainsi que son integrabilité. Tout d'abord, Vanhaecke a introduit la variété $\mathbf{M}_{g}$ de dimension supérieure à $2g+1$ définie par les coefficients des polynômes $u(x), v(x), w(x)$, leurs coefficients sont en fonction des fonctions coordonnées $x_i$ et $y_i$ \footnote{ Les fonctions coordonnées $(x_i, y_i)_{1 \leqslant i \leqslant g}$ peuvent être comme des points en position générale d'une courbe hyperelliptique.} 
 Les polynômes $u(x), v(x)$ sont définis comme plus haut dans le paragraphe 3, mais avec une différence majeure les $x_i$ et $y_i$ sont des fonctions coordonnées. Vanhaecke met une contrainte sur le polynôme $w(x)$ en le caractérisant au travers $u(x)$ et $v(x)$ par l'algorithme euclidien du quotient de l'égalité (\ref{dure}). 
On définit la structure de Poisson sur la variété $\mathbf{M}_g$ en deux étapes: La première étape consiste à établir le crochet canonique de Poisson avec $\{x_i, y_i\}=1$ et le reste des autres crochets des fonctions coordonnées sont zéro. La deuxième consiste à écrire explicitement le crochet de Poisson $\{ u(x), v(x)\}$ par un calcul en fonction des fonctions coefficients de $u(x), v(x)$, ce dernier est donné par (\ref{gea}). Les crochets de Poisson $u, v$ avec $ w$ sont plutôt délicats à obtenir car $w$ est donné par l’algorithme euclidien. Les calculs faits pour obtenir le lemme \ref{Pois1}, et les equations (\ref{5'}) et (\ref{6'}) pour définir la structure de Poisson sont extraits des travaux de Vanhaecke. 
L'espace de phase composé des polynômes $u, v$ et $w$ peut être introduit directement par l'égalité (\ref{Mg}) comme dans \cite{these} pour définir le système intégrable.\\
En prenant, $h=v^2+uw$ comme l'hamiltonien pour cette structure de Poisson. Remarquablement, on obtient le système dynamique de Mumford avec les constantes de mouvement les fonctions $(h_i)_{{0 \leqslant i \leqslant 2g-1}}$ les fonctions coefficients du polynôme $h(x)$. \\
Le système de Mumford étudié par Mumford et Vanhaecke sont des fibres au dessus de valeurs régulières par la fonction $\mathbf{H}$ définie par (\ref{H}), nous montrons dans la section \ref{section3} par théorème \ref{prop1} que les points réguliers du système de Mumford forme l'ensemble appelé la strate maximale.
La section (\ref{section2}) de cet article est un exposé détaillé de la théorie Mumford-Vanhaecke. \\

\bigskip

\subsection*{5.} Mumford a utilisé son système dynamique pour construire la jacobienne d'une 
courbe $\mathcal{C}:y^2-h(x)=0$ hyperelliptique lisse sans singularité, cela se produit seulement quand le triplet de polynômes
$u, v$ et $w$ de $M_g$ et le polynôme $h=v^2+uw$ n'admettent pas de racine commune. 
\
Soit l'application $\mathbf{H}:M_g\rightarrow \mathbb{C}[x]$ telle que $\mathbf{H}(A(x))=-\det(A(x))$.
La principale contribution originale de cet article est détendre la théorie des systèmes de Mumford aux
fibres singulières de l'application $\mathbf{H}$ (e.i aux 
courbes hyperelliptiques singulières), et nous laissons pour un article ultérieur l'étude de la
jacobienne de courbes hyperelliptiques singulières. Pour cela, 
 nous décrivons une première stratification basée sur l'analyse algébrique des fibres de $\mathbf{H}$ via le $\PGCD$ de $u(x), v(x), w(x)$ (voir proposition \ref{p30}). 
Même si cela peut sembler naturel sinon évident, la démonstration n'est pas directe et nécessite une analyse des noyaux des matrices de Toeplitz associées aux polynômes $u, v, w$.\\
Puis nous présentons une deuxième stratification géométrique des fibres à travers le degré d'indépendance des champs de vecteurs associés à $(h_i)_{{0 \leqslant i \leqslant g-1}}$ les fonctions coefficients de $\mathbf{H}$ (voir proposition \ref{st}). Cela implique une étude rigoureuse des équations de Lax (\ref{cc}) et (\ref{c'})) qui déterminent les champs de vecteurs associés à $\mathbf{H}$ et les $(h_i)_{0 \leqslant i \leqslant g-1}$. 

\
Remarquablement, nous montrons dans le théorème \ref{prop1} que ces deux stratifications coïncident. 
La preuve de la coïncidence des deux stratifications est nécessaire pour montrer 
 que la strate maximale est lisse, constituée exactement des points non singuliers et sa fermeture est formée de tous les
points singuliers 
(section \ref{lissitude})! \\

Un résultat important qui découlent du théorème \ref{t43} et du corollaire \ref{cor3.3} est que les strates non maximales d'une fibre sont isomorphes aux strates maximales d'une fibre d'un système de degré inférieur. 
Il en résulte du corollaire \ref{cor3.3} que toute strate ne contient que des points localement non singuliers 
et en particulier les strates de même degré sont équidimensionnelles. L'équidimensionalité est un fait rare, on montrera que les variétés quasi-affines de même degrés que l'on obtient sont toutes équidimensionnelles au niveau de strate. Ce dernier résultat plante le décor pour la description des 
 jacobiennes de courbes hyperelliptiques singulières, 
 cette l'étude sera présentée dans un article ultérieur. \\
 
 Tous les résultats ainsi que leurs preuves exposés dans cet article s'adaptent aux courbes hyperelliptiques $\mathcal{C}: y^2=h(x)$ avec $h$ un polynôme unitaire de degré pair.

\section{ Système de Mumford }\label{section2}
Les systèmes Hamiltoniens intégrables sont généralement définis dans le contexte de la géométrie symplectique, 
l'extension du concept d'intégrabilité la plus naturelle pour plusieurs
systèmes dynamiques est le contexte de géométrie de Poisson, où l'accent est mis sur la structure algébrique de Lie et favorise l'émergence de fonctions de Casimir qui apparaissent intrinsèquement quand la dimension de la variété de Poisson est strictement supérieure à deux fois le degré de liberté du système intégrable (voir par exemple théorème \ref {in} et l'égalité (\ref{dimg})). 

\medskip

Les systèmes de Mumford peuvent être définis par le biais de deux approches, la première approche se fait au travers du prisme des dérivations dynamiques qui a été établie par Mumford \cite{Mum}; la deuxième approche se fait à l'aide du prisme de la géométrie de Poisson qui fut exposée par Vanhaecke \cite{Pol}. Ces deux approches sont distinctes dans leur conceptualisation car la définition de la structure de Poisson est basée sur les fonctions coordonnées et les systèmes dynamiques sont basés sur l'évolution de points sur une courbe hyperelliptique. Dans cette section nous allons présenter ces deux approches et nous exposerons les liens subtils qui les lient. Commençons par le point de vue de Mumford. 

\medskip

Nous notons par $\mathbb{C}_k[x]$ l'ensemble des polynômes de degrés au plus $k$ et nous notons par $\mathbb{C}^1_k[x]$ l'ensemble des polynômes unitaires de degrés $k$. \\

\subsection{}\label{Meth1}Le motif premier de Mumford est de décrire la variété jacobienne associée à une courbe hyperelliptique et plus précisément le diviseur thêta, à cette fin il généra un système intégrable que nous allons exposer et étudier. Toutefois, la motivation de certains choix de Mumford comme l'introduction de polynômes $(u(x), v(x), w(x))$ définis par (\ref{uvx}), (\ref{uvx'}) et (\ref{uvx''}) restent mystérieuse
; c'est cet ésotérisme qui fait la beauté de ses systèmes intégrables, et qu'on essaiera de dissiper tout au long de cette section. \\

\medskip

Soit un entier naturel $g>2$. Fixons une courbe hyperelliptique $\mathcal{C}$ d'équation affine $y^2=h(x)$ où $h(x)=x^{2g+1}+\sum\limits h_ix^i$ est un polynôme de $\mathbb{C}^1_{2g+1}[x]$ avec uniquement des racines simples \footnote{ Dans cet article on focalisera notre attention uniquement sur les polynômes $h$ de degrés impairs, cependant notre étude se transpose aux polynômes $h$ de degrés pairs. }. La courbe $\mathcal{C}$ est une courbe lisse car le polynôme $h(x)$ n'a pas de racines multiples. Le genre algébrique de la courbe $\mathcal{C}$ est $g$. \\

Soit $ \mathcal{C}$ une courbe. On note par $ \mathcal{C}^g$, les $g$ copies 
de $ \mathcal{C}$. Le groupe symétrique $S_g$ agit naturellement (à droite) sur $ \mathcal{C}^g$, en permutant les facteurs. On a alors la variété 
 quotient $ \mathcal{C}^{(g)}= \mathcal{C}^g / S_g$, on appelle $ \mathcal{C}^{(g)}$ la $g$-ème puissance symétrique de $ \mathcal{C}$. 
Rappelons que la variété jacobienne d'une courbe lisse $ \mathcal{C}$ de genre $g$ est isomorphe à la variété $ \mathcal{C}^{(g)}$. \\
Toujours, dans la perspective de décrire la variété jacobienne d'une courbe hyperelliptique lisse $\mathcal{C}$ de genre $g$, Mumford a associé à tout $g$ points génériques $(p_i=(x_i, y_i))_{1 \leqslant i \leqslant g}$ de la courbe $\mathcal{C}$ ou autrement dit pour tout point générique $(p_i=(x_i, y_i))_{1 \leqslant i \leqslant g}\in \mathcal{C}^{(g)}$, deux polynômes $u(x)$ et $v(x)$ où 
\begin{align}\label{uvx}
u(x)&=\prod_{i=1}^g(x-x_i), \\ \label{uvx'}
v(x)&=\sum\limits_{i=1}^gy_i\prod_{\begin{array}{c}
i=1\\
 i\neq j\end{array}
}^g\frac{x-x_j}{x_i-x_j}. 
\end{align}
Notons que le polynôme $v(x)$ est bien défini car comme les points $(p_i)_{1 \leqslant i \leqslant g}$ sont génériques ceci implique que si $i\neq j$ alors $x_i\neq x_j$. \\
Remarquons que le polynôme $u(x)$ est unitaire de degré $g$ et le polynôme $v(x)$ est de degré au plus $g-1$, de plus pour $1 \leqslant i \leqslant g$. On remarque que $v(x_i)=y_i$, donc $v(x_i)^2-h(x_i)=0$; c'est à dire les $(x_i)_{1 \leqslant i \leqslant g}$ les racines du polynôme de $u(x)$ sont aussi des racines du polynôme $h(x)- v^2(x)$ par conséquent $h(x)- v^2(x)$ est un multiple de $u(x)$. On définit un troisième polynôme qu'on note $w(x)$ de la manière suivante:
\begin{align} \label{uvx''}
w(x)&=\frac{h(x)-v^2(x)}{u(x)}. 
\end{align}
Étant donné que, le polynôme $h(x)$ est unitaire de degré $2g+1> \deg(v(x))$ et le polynôme $u(x)$ est unitaire de degré $g$ alors $w(x)$ est un polynôme unitaire de degré $g+1$. \\
Le choix de ces trois polynômes 
 est judicieux pour représenter $g$ points génériques de la courbe hyperelliptique $\mathcal{C}$ et pour établir le système intégrable. Avant d'aller plus loin nous tenons à attirer l'attention du lecteur sur le fait que Mumford a choisi $g$ points $(p_i)_{1 \leqslant i \leqslant g}$ avec des restrictions moins rigides, ces restrictions sont présentées dans l'appendice (page \pageref{ap}). \\

\bigskip

On réécrit les trois polynômes $u(x)$, $v(x)$ et $w(x)$ sous la forme suivante: 
\begin{align}\label{uvx'''}
u(x)&=x^{g}+\sum\limits_{i=1}^{g-1}x^iu_i, \\
v(x)&=\sum\limits_{i=1}^{g-1}x^iv_i, \\
w(x)&=x^{g+1}+\sum\limits_{i=1}^{g}x^iw_i. 
\end{align}
En développant les équations (\ref{uvx}) et (\ref{uvx''}) des polynômes $u(x)$, $v(x)$ et $w(x)$ puis en identifiant les coefficients des puissances de $x$ avec (\ref{uvx'''}), on obtient que 
les coefficients $(u_j)_{0 \leqslant j \leqslant g-1}$ sont polynomiaux en fonction de $(x_i)_{1 \leqslant i \leqslant g}$
, les $(v_j)_{0 \leqslant j \leqslant g-1}$ sont rationnels en fonction de $(y_i)_{1 \leqslant i \leqslant g}$ et $(x_i)_{1 \leqslant i \leqslant g}$ 
, les $(w_j)_{0 \leqslant j \leqslant g}$ sont rationnels en fonction de $(x_i, y_i)_{1 \leqslant i \leqslant g}$ et $(h_i)_{0 \leqslant i \leqslant 2g} $. \\
On note par $M^{\mathcal{C}} $ $$M^{\mathcal{C}}=\{((u_i)_{0 \leqslant i \leqslant g-1}, (v_i)_{0 \leqslant i \leqslant g-1}), (w_i)_{0 \leqslant i \leqslant g-1})\text{ tel que } u(x)w(x)-v^2(x)-h(x)=0\}, $$

Mumford a introduit la dérivée $\frac{d}{dt}$, cette dérivée est étroitement liée à la variation des $g$ points qui définissent $u(x)$, $v(x)$ et $w(x)$ 
. Mumford \cite[page 3. 42]{Mum} a défini la dérivation $\frac{d}{dt}$ en suivant un flot tel que $\frac{dh}{dt} = 0 $. De plus l'action de dérivée $\frac{d}{dt}$ sur $u(x)$ au point $(x', y')$ comme il suit:
\begin{equation}\label{m2}{ \frac{du(x)}{dt}\Big\rvert _{(x', y')}=\frac{u(x)v(x')-v(x)u(x')}{x-x'}}, \end{equation}
On sais que $h(x)$ est un polynôme fixé et est un invariant pour la dérivée $ \frac{d}{dt}$, par conséquent 
l'action de dérivée $\frac{d}{dt}$ sur l'équation $h(x)-v^2(x)=u(x)w(x)$ au point $(x', y')$ donne: 
\begin{equation}\label{m1}
-2v(x)\frac{dv(x)}{dt}\Big\rvert _{(x', y')}=w(x) \frac{du(x)}{dt}\Big\rvert _{(x', y')} +u(x)\frac{dw(x)}{dt}\Big\rvert _{(x', y')}, 
\end{equation}
en remplaçant l'équation (\ref{m2}) dans (\ref{m1}) on obtient:
\begin{align}\nonumber
-2v(x)\frac{dv(x)}{dt}\Big\rvert _{(x', y')}=w(x)\displaystyle{\ \frac{u(x)v(x')-v(x)u(x')}{x-x'} +u(x)\frac{dw(x)}{dt}\Big\rvert _{(x', y')}}, \\\label{m3}
2v(x)\left[\displaystyle{\frac{dv(x)}{dt}\Big\rvert _{(x', y')}- \frac{w(x)u(x')}{x-x'}}\right]+u(x)\left[\displaystyle{\frac{dw(x)}{dt}\Big\rvert _{(x', y')}+ \frac{w(x)v(x')}{x-x'}}\right]=0. 
\end{align}
{Les polynômes $v(x)$ et $u(x)$ sont premiers entre eux sur un ouvert dense dans $M^{\mathcal{C}}$ alors pour que le côté gauche de l'égalité (\ref{m3}) s'annule, il faut que:}
\begin{equation}\label{m4}
\begin{array}{l}
\displaystyle{\frac{dv(x)}{dt}\Big\rvert _{(x', y')}- 
\frac{w(x)u(x')}{x-x'}}=-\frac{1}{2}u(x)a(x, x'), \\
\displaystyle{\frac{dw(x)}{dt}\Big\rvert _{(x', y')}+ \frac{w(x)v(x')}{x-x'}}=v(x)a(x, x'). \\
\end{array}
\end{equation}
On peut écrire $a(x, x')= \tilde{a}(x, x') +\frac{w(x')}{x-x'}$. En remplaçant $a(x, x')$ sous cette forme dans (\ref{m4}), on obtient:
\begin{align}\label{m6}
\frac{dv(x)}{dt}\Big\rvert _{(x', y')}= \frac{1}{2}[ \frac{w(x)u(x')-w(x')u(x)}{x-x'}-u(x). \tilde{a}(x, x')], \\ \label{m7}
\frac{dw(x)}{dt}\Big\rvert _{(x', y')}= \frac{-w(x)v(x')+w(x')v(x)}{x-x'}+v(x). \tilde{a}(x, x'). 
\end{align}
On sait que $\deg(v)<g$ par conséquent $\deg(\frac{dv(x)}{dt}\Big\rvert _{(x', y')})<g$ en $x'$, donc le côté de droite de l'égalité (\ref{m6}) doit être de degré au plus $g-1$ en ${x'}$, afin d'annuler le coefficient dominant de degré $g$ il faut choisir $\tilde{a}(x, x')=u(x')$, par conséquent les équations (\ref{m6}) et (\ref{m7}) deviennent:
\begin{equation}\label{m11} \begin{array}{l}
\displaystyle{\frac{dv(x)}{dt}\Big\rvert _{(x', y')}= \frac{1}{2}\left[ \frac{w(x)u(x')-w(x')u(x)}{x-x'}-u(x). u(x')\right]}, \\
\frac{dw(x)}{dt}\Big\rvert _{(x', y')}= \displaystyle{\frac{-w(x)v(x')+w(x')v(x')}{x-x'}}+v(x). u(x'). 
\end{array}
\end{equation}
La dérivation $\frac{d}{dt}$ est définie sur la courbe hyperelliptique $\mathcal{C}$, et est caractérisée par son action sur les trois polynômes $u(x)$, $v(x)$ et $w(x)$. La derivation $\frac{d}{dt}$ est donnée par les équations (\ref{m2}) et (\ref{m11}), définit un système dynamique, qu'on appelle le système de Mumford d'ordre $g$. Plus loin dans la partie \ref{strati12} de la section \ref{section2} , on montrera que ce système est un système intégrable. 

\bigskip

Malheureusement, nous ne comprenons pas comment Mumford a détermine le flot qui définit la derivation $\frac{d}{dt}$. 
Cependant dans cette partie, nous allons exposer la structure de Poisson introduite par Vanhaecke 
 qui convient parfaitement à une situation. Cette structure de Poisson vient à notre rescousse et nous permet de définir les hamiltoniens des systèmes de Mumford, mais cela a un prix, malgré le fait que le crochet de Poisson résultant d'un calcul facile en principe, il est plutôt 
 complexe en pratique (voir remarque \ref{rem'} )! \\

\subsection{}\label{Meth2} Vanhaecke a commencé par définir une structure de Poisson fonctions coordonnées $((x_i, y_i))_{1 \leqslant i \leqslant g}$ sur la variété $(\mathbb{C}^{2}) ^g $ comme il suit:
\begin{equation}\label{4}
\begin{array}{c}
 \{y_i, x_j\}=-\{x_j, y_i\}=1, \text {\, \, \, } \{x_i, x_j\}=\{y_i, y_j\}=0, \\
 \{y_i, x_j\}= \left\{\begin{array}{cl}
1 & \text{\, if $i=j$, }\\
0 &\text{\, otherwise. }
\end{array}\right. \\
\end{array}
\end{equation}
Le $g$-uplet des fonctions coordonnées $((x_i, y_i))_{1 \leqslant i \leqslant g}$ de $ (\mathbb{C}^{2}) ^g$ définissent deux polynômes 
\begin{equation}\label{u, v} 
\begin{array}{ll}
&u(x)=\prod\limits_{i=1}^g(x-x_i)=x^{g}+\sum\limits_{i=0}^{g-1}u_ix^i, \\ et& \\
&v(x)=\sum\limits_{i=1}^gy_i\prod\limits_{\begin{array}{c}
i=1\\
 i\neq j\end{array}
}^g\frac{x-x_j}{x_i-x_j}= \sum\limits_{i=0}^{g-1}v_ix^i.
\end{array} \end{equation}
Les coefficients du polynôme $u(x)$ sont des fonctions polynômes en les fonctions coordonnées de $(x_i)_{1 \leqslant i \leqslant g}$, Les coefficients du polynôme $v(x)$ sont rationnels en les fonctions coordonnées $(x_i, y_i)_{1 \leqslant i \leqslant g}$. \\
Dans le lemme \ref{Pois1}, nous allons détailler les crochets de Poisson des fonctions $(u_i, v_i)_{0 \leqslant i \leqslant g-1}$, qui sont assez compliqués, mais avant nous introduisons quelques notations:

\medskip

Soit $\frac{f(x)}{g(x)}$ une fonction rationnelle. On note par $\left [\frac{f(x)}{g(x)}\right]_+$ sa partie polynomiale de $\frac{f(x)}{g(x)}$ et on note par $\left [\frac{f(x)}{g(x)}\right]_-=\frac{f(x)}{g(x)}-\left [\frac{f(x)}{g(x)}\right]_+$ sa partie rationelle. Avec ces notations on a $$f(x) \mod g(x)= g(x) \left [\frac{f(x)}{g(x)}\right]_-. $$

\begin{lem}\label{Pois1}
Soit la structure de Poisson $\{\cdot, \cdot\}$ définie par les équations (\ref{4}). Les crochets de Poisson des fonctions $(u_i, v_i)_{0 \leqslant i \leqslant g-1}$ sont donnés sous la forme canonique suivante:
\begin{equation}\label{4'}
\begin{array}{l}
 \{u(x), u_j\}= \{v(x), v_j\}=0, \text {\ pour tout, \, \, } {0 \leqslant j \leqslant g-1}, \\
 \{u(x), v_j\}= -\{v(x), u_j\}=\left[ \frac{u(x)}{x^{j+1}} \right]_+. 
 \end{array}
\end{equation}

\end{lem}
\begin{proof}[Preuve]
Les polynômes $u(x)$ et $v(x)$ s'écrivent de deux manières suivantes:
\begin{align}\label{uv}
u(x)=\prod\limits_{i=1}^g(x-x_i)&\; \; \; \;, \; \;\; \; \; v(x)=\sum\limits_{i=1}^{g}y_i\prod\limits_{j=1, j \neq i}^g \frac{ x-x_j}{x_i-x_j}, \\ \label{uuvv}
u(x)=x^g+\sum\limits_{i=0}^{g-1}{u_ix^i}& \; \; \; \;, \; \;\; \; \;v(x)=\sum\limits_{i=0}^{g-1}{v_ix^i}. 
\end{align}
En développant le polynôme $u(x) $ (resp. $v(x)$) donné par (\ref{uv}) puis en identifiant les coefficients du polynôme de $u(x)$ (resp. $v(x)$) donnés par (\ref{uuvv}), 
{on obtient: }
\begin{equation*}
\begin{array}{lcl}
u_{g-j}=(-1)^{j}\sum\limits_{I \in N_{j}} \prod\limits_{t\in I} x_{t}&, & v_{g-j}=(-1)^{j}\sum\limits_{i=1}^{g}y_i \sum\limits_{I \in N^i_{j}} \prod\limits_{t\in I} \frac{ x_{t} }{x_i-x_t}. 
\end{array}
\end{equation*}
Où pour tout $0<j<g$, on note par $N_j$ (resp. $N^i_j$) l'ensemble de tous les 
sous-ensembles de $j$ éléments de $\{1, 2, 3, \cdots, g\}$ (resp. $\{1, 2, 3, \cdots, \widehat{i}, \cdots, g\}$ où $i$ est omis.). \\

\medskip
Le crochet de $\{u_k, u_j\}=0, \text {\ pour tout, \, \, } 0 \leqslant j, k \leqslant g-1$ car les $(u_j)_{0 \leqslant j, k \leqslant g-1}$ sont des polynômes en fonctions des coordonnées $(x_i)_{0 \leqslant i \leqslant g-1}$, et par définition $\{x_i, x_j\}=0$, donc 
\begin{equation}\label{u1} \{u(x), u_j\}=0, \text {\ pour tout, \, \, } {0 \leqslant j \leqslant g-1}\end{equation}

\medskip

Calculons maintenant le crochet de Poisson $\{v(x), v(y)\}$

\begin{align} \nonumber
\{v(x), v(t)\}&=\{\sum\limits_{i=1}^{g}y_i\prod\limits_{h=1, h \neq i}^g \frac{ x-x_h}{x_i-x_h}, \sum\limits_{k=1}^{g}y_k\prod\limits_{l=1, l\neq k}^g \frac{ t-x_l}{x_k-x_l}\}\\ \label{inv}
&= \sum\limits_{i=1}^{g}\sum\limits_{k=1}^{g}y_k\prod\limits_{h=1, h \neq i}^g \frac{ t-x_h}{x_i-x_h}\{y_i, \prod\limits_{l=1, l\neq k}^g \frac{ x-x_l}{x_k-x_l}\}- \sum\limits_{k=1}^{g}\sum\limits_{i=1}^{g}y_i\prod\limits_{h=1, h \neq k}^g \frac{ t-x_h}{x_k-x_h}\{y_k, \prod\limits_{l=1, l\neq i}^g \frac{ t-x_l}{x_i-x_l}\}
\end{align}
Remarquez que les deux parties composant la différence de l'équation (\ref{inv}) sont symétriques en $x, y$ et $k, i$, par conséquent on restreint le calcul à la première partie de la différence de l'équation (\ref{inv}) et on retrouve le calcul de la seconde partie par symétrie. \\
 On a 
\begin{align} \label{inv1}
\{y_i, \prod\limits_{l=1, l\neq k}^g \frac{ t-x_l}{x_k-x_l}\}=\left\{ \begin{array}{ll}
\prod\limits_{l=1, l\neq i}^g \frac{ t-x_l}{x_i-x_l}\sum\limits_{l\neq i}\frac{-1}{x_i-x_l} & \text { lorsque } i=k, \\
\prod\limits_{l=1, l\neq k}^g \frac{ t-x_l}{x_k-x_l}\frac{t-x_k}{(x_k-x_i)^2} & \text { lorsque } i\neq k, 
\end{array}\right. 
\end{align}
\\
Les équations (\ref{inv1}) sont des polynômes de degré au plus $g-1$; alors en substituant le crochet $\{y_i, \prod\limits_{l=1, l\neq k}^g \frac{ t-x_l}{x_k-x_l}\}$ donné par (\ref{inv1}) dans (\ref{inv}) le crochet $\{v(x), v(t)\}$
. Les égalités (\ref{inv1}) impliquent que $\{v(x), v(t)\}$ s'annule en $2g$ points suivants $(x_i, x_k)_{1 \leqslant i \neq k \leqslant g}$, mais comme $\{v(x), v(t)\}$ est un polynôme de degré au plus $2(g-1)$ en $(x, t)$, alors \begin{equation} \label{v1}\{v(x), v(t)\}=0. \end{equation}
De \ref{u1} et \ref{v1} on a prouvé la première égalité de (\ref{4'}).

\medskip

Il ne reste plus qu'à calculer le crochet de Poisson $\{u_{g-j}, v(x)\}$. 
\begin{align} \nonumber
\{u_{g-j}, v(x)\}&=\{(-1)^{j}\sum\limits_{I \in N_{j}} \prod\limits_{t\in I} x_{t}, \sum\limits_{l=1}^{g}y_l\prod\limits_{h=1, h \neq l}^g \frac{ x-x_h}{x_l-x_h}\}, \\ \nonumber
&=(-1)^{j} \sum\limits_{I \in N_{j}}\sum\limits_{l=1}^{g} \{ \prod\limits_{t\in I} x_{t}, y_l\}\prod\limits_{h=1, h \neq l}^g \frac{ x-x_h}{x_l-x_h}, \\ \label{rep}
&= (-1)^{j-1}\sum\limits_{l=1}^{g}\sum\limits_{I \in N^l_{j-1}}\prod\limits_{t\in I} x_{t}\prod\limits_{h=1, h \neq l}^g \frac{ x-x_h}{x_l-x_h}. 
\end{align}
Remarquons que 
\begin{align*}
\sum\limits_{I \in N^l_{j-1}}\prod\limits_{t\in I} x_{t}&= \sum\limits_{k=0}^{j-1}(-x_l)^k\sum\limits_{I \in N_{k}}\prod\limits_{t\in I} x_{t}, \\
&=(-1)^{j-1}\sum\limits_{k=0}^{j-1}(-1)^{k +j-1} x_l^k\sum\limits_{I \in N_{k}}\prod\limits_{t\in I} x_{t}, \\
&=(-1)^{j-1}\sum\limits_{k=0}^{j-1}x_l^{k }u_{g-j+1+k}. 
\end{align*}
En remplaçant cette dernière égalité dans (\ref{rep}), on obtient 
\begin{align} \nonumber
\{u_{g-j}, v(x)\}
&= \sum\limits_{l=1}^{g}\sum\limits_{k=0}^{j-1}x_l^{k }u_{g-j+1+k}\prod\limits_{h=1, h \neq l}^g \frac{ x-x_h}{x_l-x_h}, 
\end{align}
On a l'égalité suivante \begin{align}\label{gea}\sum\limits_{l=1}^{g}\sum\limits_{k=0}^{j-1}x_l^{k }u_{g-j+1+k}\prod\limits_{h=1, h \neq l}^g \frac{ x-x_h}{x_l-x_h}= \left[\frac{u(x)}{x^{g-j+1}}\right]_+. \end{align}
 parce que les deux parties de l'égalité (\ref{gea}) sont deux polynômes de degré au plus $g-1$ et leurs évaluations en $g$ points $(x_i)_{0 \leqslant i \leqslant g}$ sont égales. 
Rappelons si deux polynômes de degré au plus $g-1$ sont égaux en $g$ points alors
les deux polynômes sont identiques
, par conséquent l'égalité (\ref{gea}) est correcte. 
D'où 
\begin{align*}
\{u_{g-j}, v(x)\}&= \left[ \frac{u(x)}{x^{g-j+1}} \right]_+, 
\end{align*}
en faisant le changement de variable $i=g-j$ on retrouve bien la deuxième égalité des équations (\ref{4'}), ce qu'il fallait démontrer.
\end{proof}

\bigskip
 Les polynômes $u(x)$ et $v(x)$ présentés par Vanhaecke sont des polynômes en fonction des fonctions coordonnées $(u_i)_{0 \leqslant i \leqslant g-1 } $ et $(v_i)_{0 \leqslant i \leqslant g-1 }$ respectivement. Afin de définir un troisième polynôme $w(x)$, on choisit un polynôme quelconque $f(x)\in\mathbb{C}[x]$ et on définit un polynôme $F(x, y)=y^2-f(x)$. Le polynôme $F(x, y)$ est appelé hyperelliptique car $F(x, y)=0$ est l'équation d'une courbe hyperelliptique. On définit deux polynômes 
\begin{equation}\label{dure}H(u(x), v(x))=F(x, v(x))\mod u(x) \text{ \, et\, } w(x)= -\left[ \frac{F(x, v(x))}{u(x)} \right]_+. 
\end{equation}
Nous allons écrire explicitement les actions de $u(x)$ et $v(x)$ sur $w(x)$ par le crochet de Poisson:
\begin{align}\nonumber
\{u_i, w(x)\}&= \{u_i, -\left[ \frac{F(x, v(x))}{u(x)} \right]_+\}, \\ \nonumber
&=\left[ \frac{\{u_i, -F(x, v(x))\}}{u(x)} \right]_+, \\ \nonumber
&=\left[ \frac{\{u_i, - v(x)^2\}}{u(x)} \right]_+, \\ \label{5'}
&=2\left[ \frac{-v(x)\left[ \frac{u(x)}{x^{i+1}} \right]_+}{u(x)} \right]_+, 
\end{align}
et 
\begin{align} \nonumber
\{v_i, w(x)\}&= \{v_i, -\left[\frac{F(x, v(x))}{u(x)} \right]_+\}, \\ \nonumber
&=\left[ \frac{\{v_i, u(x)\} F(x, v(x))}{u^2(x)} \right]_+, \\ \nonumber
&=\left[ \frac{\left[ \frac{u(x)}{x^{i+1}} \right]_+ F(x, v(x))}{u(x)^2} \right]_+, \\ \label{6'}
&=\left[- \frac{\left[ \frac{u(x)}{x^{i+1}} \right]_+ }{u(x)}w(x) \right]_+. 
\end{align}


\begin{coro}\label{idc}
Les crochets de Poisson de $u(z)$, $v(z)$ et $w(z)$ avec $H(u(x), v(x))$ sont égaux à
\begin{equation}\label{huvw}
\begin{array}{ll}
\{u(z), H(u(x), v(x))\}&=2\sum\limits_{i=0}^{g-1}x^{i} \left(v(z)\left[ \frac{u(z)}{z^{i+1}} \right]_+ - u(z)\left[\frac{v(z)}{u(z)}\left[ \frac{u(z)}{z^{i+1}} \right]_+ \right]_+\right), \\
\{v(z), H(u(x), v(x))\}&=\sum\limits_{i=0}^{g-1}-x^i \left({w(z)}\left[ \frac{u(z)}{z^{i+1}} \right]_+- \left[\frac{w(z)}{u(z)}\left[ \frac{u(z)}{z^{i+1}} \right]_+\right]_+\right), 
\\
\{w(z), H(u(x), v(x))\}&=2\sum\limits_{i=0}^{g-1}x^i \left(w(z) \left[\frac{v(z)}{u(z)}\left[ \frac{u(z)}{z^{i+1}} \right]_+\right]_+- v(z)\left[\frac{w(z)}{u(z)}\left[ \frac{u(z)}{z^{i+1}} \right]_+\right]_+\right).
\end{array} 
\end{equation}
\end{coro}
\begin{proof}[Preuve]
Commençons par calculer le crochet $\{u(z), H(u(x), v(x))\}$
\begin{align}\nonumber
\{u(z), H(u(x), v(x))\}&=\sum\limits_{j=0}^{g-1}\{u(z), v_j\}\frac{\partial H(u(x), v(x))}{\partial v_j}, \\ \nonumber
&=\sum\limits_{j=0}^{g-1}\{u(z), v_j\}\frac{\partial H(u(x), v(x))}{\partial v_j}, \\ \nonumber
&=\sum\limits_{j=0}^{g-1} \left[ \frac{u(z))}{z^{j+1}} \right]_+\frac{\partial H(u(x), v(x))}{\partial v_j}\\\nonumber
&=\sum\limits_{j=0}^{g-1} \sum\limits_{k=0}^{g-j-1} z^{g-j-k-1}u_{g-k} \frac{\partial F(x, v(x))}{\partial y} x^j\mod u(x)\\\nonumber
&=\sum\limits_{l=1}^{g} \sum\limits_{j=0}^{g-l} z^{l-1}u_{j+l}\frac{\partial F(x, v(x))}{\partial y} x^j\mod u(x)\\\nonumber
&=\sum\limits_{l=1}^{g} \sum\limits_{j=0}^{g-l} z^{l-1}u_{j+l}\frac{\partial F(x, v(x))}{\partial y} x^j\mod u(x)\\ \label{hu}
&=\sum\limits_{i=0}^{g-1}z^{i} 2[v(x)\left[ \frac{u(x)}{x^{i+1}} \right]_+ - u(x)\left[\frac{v(x)}{u(x)}\left[ \frac{u(x)}{x^{i+1}} \right]_+ \right]_+]
\end{align}
Nous déterminons le crochet $\{v(z), H(u(x), v(x))\}$:
\begin{align}\label{+=}
\{v(z), H(u(x), v(x))\}&=\sum\limits_{j=0}^{g-1}\{v(z), u_j\}\frac{\partial H(u(x), v(x))}{\partial u_j}, 
\end{align}
pour alléger les équations nous allons calculer $\frac{\partial H(u(x), v(x))}{\partial u_j}$ à part:
\begin{align*}
\frac{\partial H(u(x), v(x))}{\partial u_j}&=\partial \frac{F(x, v(x))\mod u(x)}{\partial u_j}, \\
&= \frac{\partial}{\partial u_j} [u(x)\left[\frac{F(x, v(x))}{u(x)}\right]_- ], \\
&=\frac{\partial}{\partial u_j} [F(x, v(x))-u(x)\left[\frac{F(x, v(x))}{u(x)}\right]_+], \\
&=\frac{\partial}{\partial u_j} [-u(x)\left[\frac{F(x, v(x))}{u(x)}\right]_+], \\
&=-x^i\left[\frac{F(x, v(x))}{u(x)}\right]_+-u(x)\left[\frac{-x^i}{u(x)}\frac{F(x, v(x))}{u(x)}\right]_+, \\
&=-u(x)[\frac{x^i}{u(x)}\left[\frac{F(x, v(x))}{u(x)}\right]_+-\left[\frac{x^i}{u(x)}\frac{F(x, v(x))}{u(x)}\right]_+], \\
&=-x^i\left[\frac{F(x, v(x))}{u(x)}\right]_+\mod u(x). 
\end{align*}
En remplaçant cette dernière égalité dans (\ref{+=}), on obtient:
\begin{align} \nonumber 
\{v(z), H(u(x), v(x))\}&=\sum\limits_{j=0}^{g-1}\{v(z), u_j\}[-x^j\left[\frac{F(x, v(x))}{u(x)}\right]_+\mod u(x)], \\ \nonumber 
&=\sum\limits_{i=0}^{g-1}-x^i\left[\frac{F(z, v(z))}{u(z)}\right]_+\left[ \frac{u(z)}{z^{i+1}} \right]_+\mod u(z), \\ \nonumber 
&=\sum\limits_{i=0}^{g-1}-x^i w(z)\left[ \frac{u(z)}{z^{i+1}} \right]_+\mod u(z), \\ \nonumber 
&=\sum\limits_{i=0}^{g-1}-x^i {w(z)}\left[ \frac{u(z)}{z^{i+1}} \right]_+\mod u(z), \\ \label{hv}
&=\sum\limits_{i=0}^{g-1}-x^i ({w(z)}\left[ \frac{u(z)}{z^{i+1}} \right]_+- \left[\frac{w(z)}{u(z)}\left[ \frac{u(z)}{z^{i+1}} \right]_+\right]_+). 
\end{align}
Nous exposons maintenant le développement du crochet $\{w(z), H(u(x), v(x))\}$
\begin{align} \nonumber 
\{w(z), H(u(x), v(x))\}&=\{\left[\frac{F(z, v(z))}{u(z)}\right]_+, H(u(x), v(x))\}, \\ \nonumber 
&=\{\left[\frac{v(z)^2-f(z)}{u(z)}\right]_+, H(u(x), v(x))\}, \\ \nonumber 
&=\left[\frac{2v(z)\{v(z), H(u(z), v(z))\}}{u(x)}\right]_++\left[{F(z, v(z))}\frac{\{u(z), H(u(x), v(x))\}}{{u(z)}^2}\right]_+, \\ \nonumber 
&=2\left[\sum\limits_{i=0}^{g-1}-x^i\frac{v(z)}{u(z)}[{w(z)}\left[ \frac{u(z)}{z^{i+1}} \right]_+-u(z) \left[\frac{w(z)}{u(z)}\left[ \frac{u(z)}{z^{i+1}} \right]_+\right]_+]\right]_+\\ \nonumber 
&+\left[\frac{F(z, v(z))}{u(z)}\frac{\{u(z), H(u(x), v(x))\}}{{u(z)}}\right]_+, \\ \nonumber 
&=2\left[\sum\limits_{i=0}^{g-1}-x^i\frac{v(z)}{u(z)} [{w(z)}\left[ \frac{u(z)}{z^{i+1}} \right]_+ - \left[\frac{w(z)}{u(z)}\left[ \frac{u(z)}{z^{i+1}} \right]_+\right]_+]\right]_+, \\ \nonumber 
&+\left[w(z)\sum\limits_{i=0}^{g-1}2x^{i} \left(\frac{v(z)}{{u(z)}}\left[ \frac{u(z)}{z^{i+1}} \right]_+ -\left[\frac{v(z)}{u(z)}\left[ \frac{u(z)}{z^{i+1}} \right]_+ \right]_+\right)\right]_+, \\ \label{hw}
&=2\sum\limits_{i=0}^{g-1}x^i [ w(z) \left[\frac{v(z)}{u(z)}\left[ \frac{u(z)}{z^{i+1}} \right]_+\right]_+- v(z)\left[\frac{w(z)}{u(z)}\left[ \frac{u(z)}{z^{i+1}} \right]_+\right]_+]. 
\end{align}

le calcul des crochets de $ \{u(z), H(u(x), v(x))\}$, $ \{v(z), H(u(x), v(x))\}$ et $ \{w(z), H(u(x), v(x))\} $ nous donne (\ref{hu}) , (\ref{hv}) et (\ref{hw}) respectivement. Ces résultats sont identiques aux égalités (\ref{huvw}) ce qu'il fallait démontrer.
\end{proof}
\begin{rem}\label{rem'}
Les derniers $g$ coefficients du polynôme $F(x, v(x)) $ sont indépendants car le coefficient $v_i$ de $v(x)$ n'apparaît que sur le coefficient de $(x^j) _{i \leqslant j \leqslant g-1 }$ de $F(x, v(x))$, alors les coefficients $(h_i)_{ 0 \leqslant i \leqslant g-1}$ du polynôme $H(u(x), v(x))=F(x, v(x))\mod u(x) $ sont l'indépendants. 

\end{rem}
Soient les polynômes $F(x, y)$ et $F'(x, y)$ avec
$$F(x, y)= F'(x, y)+c(x), $$
où $c(x)$ qui est un polynôme indépendant de $y$ et de degré inférieur à $g-1$ dans $x$, alors
$$F(x, y) \mod u(x)= F'(x, y) \mod u(x)+c(x), $$
alors le crochet de Poisson $\{\cdot, c(x)\}=0$, car le fix polynôme $c(x)$ est indépendant des polynômes $v(x), u(x)$. \\ 

Avant de progresser dans notre étude nous allons donner deux exemples du calcule des polynômes $w$ et $H$ pour $g=1$. \\
Choisissons $F(x, y)=y^2$ c'est à dire $f(x)=0$. On a par définition 
\begin{equation*}
\begin{array}{cl c cl}
w(x)&=\left[ \frac{F(x, v(x))}{u(x)}\right]_+&, \, & H(u(x), v(x))&=\left[ \frac{F(x, v(x))}{u(x)}\right]_-, \\
&=\left[ \frac{v^2(x)}{u(x)} \right]_+&, \, & &=u(x)\left[\frac{F(x, v(x))}{u(x)}-\left[ \frac{F(x, v(x))}{u(x)}\right]_+\right], \\
&=\left[ \frac{v_1^2x^2+v_1v_0x+v^2_0}{x^2+u_1x+u_0} \right]_+ &, \, & &=F(x, v(x))-u(x)\left[ \frac{F(x, v(x))}{u(x)}\right]_+
\end{array}
\end{equation*}
par la division euclidienne on obtient 
\begin{equation*}
\begin{array}{cl c cl}
w(x)&=v_1^2 &, \, & H(x)&=(2v_1v_0-u_1v_1^2)x+v^2_0-u_0v_1^2
\end{array}
\end{equation*}
Maintenant choisissons $f(x)=-x^3$ on a $F(x, y)=y^2+x^3$ toujours en utilisant la division euclidien on a 
$$\begin{array}{cl c cl}
w(x)&=x-v_1^2-u_1 &, \, & H(u(x), v(x))&=[2v_1v_0+u_0-u_1(v_1^2-u_1)]x+v^2_0-(v_1^2-u_1)u_0. 
\end{array}$$
Ces deux exemples $f(x)=0$ et $f(x)=-x^3$ on constate que pour écrire les polynômes $w$ et $H$ demande un effort, alors pour détourner cette difficulté et homogénesie l'écriture des polynômes $w$ et $H$ nous allons imposer une condition au polynôme $f$ incarnée par l'égalité (\ref{huv}) qu'on verra plus loin. 
\medskip

Rappelons la définition d'un système intégrable:
\begin{defi} \label{defi_sys} 
Soit $(V, \{\cdot, \cdot\})$ un espace affine de dimension $d$, muni d'une structure de Poisson de rang $2r$. Soit
$\mathbf{F}=(F_1, \dots, F_s)$ une famille de fonctions de $\mathcal{F}(V)$. Le triplet $(V, 
\{\cdot, \cdot\}, \mathbf{F}) $ est un système intégrable (au sens de Liouville) de rang $2r$ s'il satisfait les trois conditions suivantes:\\ 
$\bullet$ $s=d-r$, \\
$\bullet$ $\mathbf{F}$ est involutive, (i.e pour tous $F, G \in \mathbf{F}$ on a $\{F, G\}=0$)\\
$\bullet$ $ \mathbf{F}$ est indépendante (i.e les champs de vecteurs associés aux fonctions de $ \mathbf{F}$ sont linéairement indépendants sur un ouvert dense). \\
La famille de fonctions $\mathbf{F}=(F_1, \dots, F_s)$ de $\mathbb{C}(V)$, est vue comme l'application $\mathbf{F}: V
\longrightarrow \mathbb{C}^s$. 
\end{defi}

\begin{theo}\label{in}
Soit $M_{g, f}$ l'ensemble des triplets de polynômes $(u(x), v(x), w(x))$ de dimension $3g+1$, et $H(u(x), v(x))=\sum\limits_{i=0}^{g-1}h_ix^i$. Le triplet $(M_{g, f}, \{\cdot, \cdot\}, H) $ est un système intégrable (au sens de Liouville) de rang $2g$. 
\end{theo}
\begin{proof}[Preuve]
Par construction, la structure de Poisson définie par (\ref{4}) est de rang $2g$. Comme on l'a vu dans la remarque \ref{rem'} les coefficients $(h_i)_{ 0 \leqslant i \leqslant g-1}$ sont linéairement indépendants car la fonction coefficient $v_i$ de $v(x)$ n'apparaît que sur le coefficient de $(x^j) _{i \leqslant j \leqslant g-1 }$ et par la construction du crochet de Poisson on a que les champs de vecteurs $(\{\cdot, h_i\})_{ 0 \leqslant i \leqslant g-1}$ sont linéairement independents sur un ouvert dense. Une autre preuve de l'independence linaire des champs de vecteurs $(\{\cdot, h_i\})_{ 0 \leqslant i \leqslant g-1}$, se trouve plus bas dans le théorème \ref{prop1}.
Pour montrer que la famille $(h_i)_{0 \leqslant i \leqslant g-1}$ est involutive, il suffit de montrer que $\{H(u(x), v(x)), h_i\}=0$ pour tout $0 \leqslant i \leqslant g-1 $. 
\begin{align*}
\{H(u(x), v(x)), h_i\}&=\{u(x)\left[ \frac{F(x, v(x))}{u(x)} \right]_-, h_i\}, \\
&=\{u(x), h_i\}\left[ \frac{F(x, v(x))}{u(x)} \right]_-
+u(x)(\left[ \frac{\{F(x, v(x)), h_i\}}{u(x)}-\frac{\{u(x), h_i\}F(x, v(x))}{u^2(x)} \right]_-), \\
&=2\left(v(x)\left[ \frac{u(x)}{x^{i+1}} \right]_+ - u(x)\left[\frac{v(x)}{u(x)}\left[ \frac{u(x)}{x^{i+1}} \right]_+ \right]_+\right)\left[ \frac{F(x, v(x))}{u(x)} \right]_-\\
&+u(x)(\left[ \frac{\{F(x, v(x)), h_i\}}{u(x)}-\frac{ \left(v(x)\left[ \frac{u(x)}{x^{i+1}} \right]_+ - u(x)\left[\frac{v(x)}{u(x)}\left[ \frac{u(x)}{x^{i+1}} \right]_+ \right]_+\right)F(x, v(x))}{u^2(x)} \right]_-), \\
&=2u(x)\left(\frac{v(x)}{u(x)}\left[ \frac{u(x)}{x^{i+1}} \right]_+ - \left[\frac{v(x)}{u(x)}\left[ \frac{u(x)}{x^{i+1}} \right]_+ \right]_+\right)\left[ \frac{F(x, v(x))}{u(x)} \right]_-\\
&+u(x)(\left[ 2v(x) \left[ \left[\frac{u(x)}{x^{i+1}}\right]_+\frac{F(x, v(x))}{u(x)}\right]_--2\left(\frac{v(x)}{u(x)}\left[ \frac{u(x)}{x^{i+1}} \right]_+ - \left[\frac{v(x)}{u(x)}\left[ \frac{u(x)}{x^{i+1}} \right]_+ \right]_+\right) \frac{F(x, v(x))}{u(x)} \right]_-), 
\end{align*}
\text{ en annulant les parties polynomiales se trouvant entre les crochets $[\frac{\, \star \, }{\, \star \, }]_-$, on obtient: }
\begin{align*}
\{H(u(x), v(x)), h_i\}&=2u(x)\left(\frac{v(x)}{u(x)}\left[ \frac{u(x)}{x^{i+1}} \right]_+ - \left[\frac{v(x)}{u(x)}\left[ \frac{u(x)}{x^{i+1}} \right]_+ \right]_+\right)\left[ \frac{F(x, v(x))}{u(x)} \right]_-\\
&+u(x)\left[-2\left(\frac{v(x)}{u(x)}\left[ \frac{u(x)}{x^{i+1}} \right]_+ - \left[\frac{v(x)}{u(x)}\left[ \frac{u(x)}{x^{i+1}} \right]_+ \right]_+\right) \frac{F(x, v(x))}{u(x)} \right]_-, \\
&=2u(x)\left(\frac{v(x)}{u(x)}\left[ \frac{u(x)}{x^{i+1}} \right]_+ - \left[\frac{v(x)}{u(x)}\left[ \frac{u(x)}{x^{i+1}} \right]_+ \right]_+\right)\left[ \frac{F(x, v(x))}{u(x)} \right]_+\\
&+u(x)\left[-2\left(\frac{v(x)}{u(x)}\left[ \frac{u(x)}{x^{i+1}} \right]_+ - \left[\frac{v(x)}{u(x)}\left[ \frac{u(x)}{x^{i+1}} \right]_+ \right]_+\right) \frac{F(x, v(x))}{u(x)} \right]_+, \\
&=0. 
\end{align*}
On a donc $(\{H(u(x), v(x)), h_i\}= \sum\limits_{j=0}^{2g+1}x^j\{h_j, h_i\} =0)_{0 \leqslant i \leqslant g-1}$. Comme $x$ est uns variable formelle alors les coefficients du polynôme $\{H(u(x), v(x)), h_i\}$ sont tous nuls, par consequent $\{h_j, h_i\} =0)_{0 \leqslant i, j \leqslant g-1}$ et on conclut que la famille $(h_i)_{0 \leqslant i \leqslant g-1}$ est involutive.
\end{proof}
Vanhaecke a réinterprété le système intégrable $(M_{g, f}, \{\cdot, \cdot\}, H)$ avec une écriture plus compacte. 
Soit $f(x)\in \mathbb{C}[x]_{2g+1}^1$, alors le polynôme $w(x)$ de $M_{g, f}$ est un polynôme unitaire de degré $g+1$. 
l'espace affine noté $M_{{g, f}}$ de coordonnées $(u_i, v_j, w_k)$ avec $ 0 \leqslant i, j \leqslant g-1$ et $ 0 \leqslant k \leqslant g$ et de dimension $3g+1$. L'espace $M_{{g, f}}$ admet la structure de Poisson $\{\cdot, \cdot\}$, définie par (\ref{4'}), (\ref{5'}), et (\ref{6'}). 
Vanhaecke \cite{Pol} a présenté l'espace $M_{g, f}$ sous la forme d'un sous espace de $\mathfrak{sl}(\mathbb{C}[x])$ de la façon suivante:
\begin{equation} \label{Mg} M_{g, f}:=\left\lbrace \left(\begin{array}{cc}
v(x) & u(x) \\
w(x) & -v(x)
\end{array} \right) 
| \;u(x) \in \mathbb{C}_g^1[x], \, v (x) \in \mathbb{C}_{g-1}[x], \, w(x) \in \mathbb{C}[x]_{2g+1}^1 
\right\rbrace \simeq\mathbb{C}^{3g+1}. 
\end{equation}
%

Soit $A(x)= \left(\begin{array}{cc}
v(x) & u(x) \\
w(x) & -v(x)
\end{array} \right)\in M_{g, f}$ on a 
\begin{align}\nonumber
\det(A(x)-y\id)&=y^2-v^2(x)-u(x)w(x), \\ \nonumber
&=y^2-v^2(x)-u(x)\left[\frac{v^2(x)-f(x)}{u(x)}\right]_+, \\ \nonumber
&=y^2+f(x)+u(x)\left[\frac{v^2(x)-f(x)}{u(x)}\right]_-, \\ \label{c2}
&=y^2+f(x)+H(u(x), v(x)). 
\end{align}
De l'égalité (\ref{c2}) implique que $ \det(A(x))=f(x)+H(u(x), v(x)). $
Différents polynômes hyperelliptiques ne changent rien à la structure de Poisson. 
Soit $A(x)\in M_{g, f}$ 
 \begin{align}\label{huv} f'(x)&= \det(A(x))+ x^g \left[\frac{\det(A(x))}{x^g}\right]_-, 
 \end{align}
alors $A(x)\in M_{g, f'}$ avec $f'(x)\in \mathbb{C}[x]_{2g+1}^1$. 
On note par $M_g = 
\bigsqcup\limits_
{f \in \mathbb{C}[x]_{2g+1}^1} 
M_{g, f}$. \\

Nous allons réécrire la structure de Poisson de manière plus abrégée à l'aide du corollaire suivant qui est dû à L. Makar-Limanov \cite{L}
\begin{coro}\label{lenny}
Soit deux polynômes $a(x)=\sum\limits_{i=1}^{m}x^ia_i$ et $b(x)=\sum\limits_{i=1}^{n}x^ib_i$, et soit le polynôme 
 $T(x, y)=\frac{a(x)b(y)-a(y)b(x)}{x-y}$. Le coefficient de $y^r$ de $T$ est 
 $$\left[ \frac{b(x)}{x^{u+1}} \right]_+b(x) -\left[ \frac{b(x)}{x^{u+1}} \right]_+a(x). $$
\end{coro}
\begin{proof}[Preuve]
Supposons que $a(x)$ et $b(x)$ sont des monômes $x^r$ et $x^s$. \\
$$\frac{x^ry^s-x^s y^r}{x-y}=
\left\{
\begin{array}{ll}
x^s y^s\sum\limits_{i=1}^{r - s}x{r - s - i}y^i& \text{ si } r>s, \\
-x^r y^r\sum\limits_{i=1}^{s-r}x^{s-r-i}y^i& \text{ si } r<s, 
\end{array}
\right. $$
le coefficient de $y^u$ de $a_r b_s\frac{x^r y^s-x^s y^r}{x-y}$ 
est $a_r b_s x^{r+s-1-u }$ si $r-1\geqslant u \geqslant s$ et $a_rb_sx^{r+s-1-u }$ si $s-1\geqslant u \geqslant r. $\\
Le coefficient icient de $y^r$ de $T(x, y)$ est 
 $$\left[ \frac{a(x)}{x^{u+1}} \right]_+(b(x) -\left[ \frac{b(x)}{x^{u+1}} \right]_+x^{u+1}) - \left[ \frac{b(x)}{x^{u+1}} \right]_+(a(x) - [ \left[ \frac{a(x)}{x^{u+1}} \right]_+ ]x^{u+1}) = \left[ \frac{b(x)}{x^{u+1}} \right]_+b(x) - \left[ \frac{b(x)}{x^{u+1}} \right]_+a(x). $$

\end{proof}
À l'aide du corollaire \ref{lenny}, on peut réécrire le crochet de Poisson des fonctions coordonnées de $M_g$ définies par le lemme \ref{Pois1} à l'aide de l'écriture canonique suivante \footnote{Ici $x$ et $y$ sont des paramètres formels. }:
\begin{equation}\label{21*}
\begin{array}{cll}
 \{u(x), u(y)\} & =&\{v(x), v(y)\} = 0, \\ 
 \{u(x), v(y)\} & = &\displaystyle {\frac{u(x)-u(y)}{x-y}}, \\ 
 \{u(x), w(y)\} & =&-2\displaystyle{ \frac{v(x)-v(y)}{x-y}}, \\ 
 \{v(x), w(y)\} & = &\displaystyle {\frac{w(x)-w(y)}{x-y}-u(x)}, \\ 
 \{w(x), w(y)\} & = &2(v(x)-v(y)). 
 \end{array}
\end{equation}

Soit l'application $\mathbf{H}^g$ allant de $M_g$ vers les polynômes unitaires de degré $2g+1$
\begin{equation}\label{H}\begin{array}{cccl}
 \mathbf{H}^g:&M_{g}& \longrightarrow & \mathbb{C}^1_{2g+1}[x]\\
& A(x)=\left(\begin{array}{cc}
v (x)&{u(x)} \\
w (x)& -v(x)
\end{array} \right)& \longrightarrow & -\det \left(\begin{array}{cc}
v (x)&{u(x)} \\
w (x)& -v(x)
\end{array} \right)=
v(x)^2+u(x)w(x)
\end{array}\end{equation}
s'il n'y a pas d'ambiguïté on omettra l'indice $g$ de $ \mathbf{H}^g$. 
L'application $\mathbf{H}$ définit deux collections de fonctions de $M_g$ vers $\mathbb{C}$. La première collection est la famille finie composée des fonctions coefficients $\{h_i\}_{0 \leqslant i \leqslant 2g+1} $de $\mathbf{H}$:
$$\mathbf{H}(A(x))=\sum\limits_{i=1}^{2g+1}h_i(A(x))x^i. $$
Le seconde collection est constituée d'une famille infinie de fonctions $\{\mathbf{H}_z\}_{z\in \mathbb{C}}$ qui sont l'évaluation de l'application $\mathbf{H}$ au point $z\in \mathbb{C}$: 
$$\begin{array}{cccc}
 \mathbf{H}_z:&M_{g}& \longrightarrow & \mathbb{C}\\
& \left(\begin{array}{cc}
v &{u} \\
w & -v
\end{array} \right)& \longrightarrow &
v(z)^2+u(z)w(z)
\end{array}$$
Les fonctions $\{\mathbf{H}_z\}_{z\in \mathbb{C}}$ s'expriment à l'aide de $\{h_i\}_{0 \leqslant i \leqslant 2g+1} $ de la manière suivante: $$\mathbf{H}_z(A)=\sum\limits_{i=0}^{2g+1}h_i(A)z^i. $$
Les deux familles de fonctions $\{\mathbf{H}_z\}_{z\in \mathbb{C}}$ et $\{h_i\}_{0 \leqslant i \leqslant 2g+1} $ définissent à l'aide de la structure de Poisson deux familles de champs Hamiltoniens $\{D^{M_g}_z\}_{z\in \mathbb{C}}$ et $\{D^{M_g}_i\}_{0 \leqslant i \leqslant 2g+1} $ respectivement de la manière suivante:
\begin{eqnarray}\label{28}
D^{M_g}_z&=& \{\cdot, \mathbf{H}_z\}, \text{\, où } z\in\mathbb{C}, \\
D^{M_g}_i&=& \{\cdot, h_i\}, \text{\, avec } i\in[0, 2g+1]. 
\end{eqnarray}
On omettra d'écrire l'indice $M_g$ de $ D^{M_g}_z$ et $D^{M_g}_i $ s'il n'y a pas d'ambiguïté. 
Les champs de vecteurs $D_z$ pour tout $z\in \mathbb C$, à l'aide du crochet de Lie:
\begin{align}
D_z\vert_{A(x)}&= \left[A(x), -\displaystyle{\frac{A(z)}{x-z}}- \left(\begin{array}{cc}
0 & 0 \\ 
u(z) & 0
\end{array} \right)\right], \label{cc}
\end{align}
 Le champ de vecteur $D_z\vert_{A(x)}$ défini par l'égalité (\ref{cc}) est une matrice dont les composantes sont des polynômes en $z$ de degré au plus $g-1$ par conséquent $D_z\vert_{A(x)}=\sum\limits_{i=0}^{g-1}z^iD_i\vert_{A(x)}$ \footnote{ Notez que si $z=0$, on établit $z^0=1$.}. 
Cependant par définition $D_z\vert_{A(x)}=\sum\limits_{i=0}^{2g+1}D_i\vert_{A(x)}z^i$ alors 
\begin{align}\label{dimg}
D^{M_g}_i= \{\cdot, h_i\}=0, \text{\, pour tout } g \leqslant i \leqslant 2g+1. 
\end{align}
Les fonctions $(h_i)_{ g \leqslant i \leqslant 2g+1}$ sont appelées des fonctions de Casimir pour la structure de Poisson. \\
Soit ${0 \leqslant i \leqslant g-1}$, l'égalité (\ref{dimg}) entraîne que le champ de vecteur $D_i$ correspond au coefficient du monôme $z^i$ de (\ref{cc}). Par conséquent, on peut écrire les champs $(D_i)_{0 \leqslant i \leqslant g-1}$ à l'aide du crochet de Lie de la manière suivante:
\begin{align}
D_i\vert_{A(x)}&=\left[A(x), \left[ \displaystyle{\frac{A(x)}{x^{i+1}}}\right]_+ - \left(\begin{array}{cc}
0 & 0 \\ 
u_i & 0
\end{array} \right)\right], \label{c'}
\end{align}
où $\left[ {\frac{A(x)}{x^{i+1}}}\right]_+$ est la partie polynomiale de la matrice $ {\frac{A(x)}{x^{i+1}}}$. les équations (\ref{cc}) et (\ref{c'}) sont appelés les équations de Lax. 

À l'aide des égalités \ref{21*}, les crochets de Poisson de $\mathbf{H}_{x'}$ avec $u(x), v(x)$ et $w(x)$ s'écrivent comme il suit:
\begin{align}\begin{array}{cll}
\{u(x), \mathbf{H}_{x'}\} &=\{u(x), v^2(x')+u(x')w(x')\} & =2\displaystyle{\frac{u(x)v(x')-u(x')v(x)}{x-x'}}, \\ 
\{v(x), \mathbf{H}_{x'}\} &=\{v(x), v^2(x')+u(x')w(x')\} & =\displaystyle{\frac{u(x')w(x)+w(x')u(x)}{x-x'}}-u(x')u(x), \\ 
\{w(x), \mathbf{H}_{x'}\} &=\{w(x), v^2(x')+u(x')w(x')\}&= 2\displaystyle{\frac{v(x)w(x')-w(x)v(x')}{x-x'}}+2v(x)u(x'). 
\end{array} \label{eg}
\end{align}

\subsection{}\label{part3}Notons que l'action de la différentielle $\frac{d}{dt}\Big\rvert _{(x', y')}$ sur les polynômes $u(x), v(x)$ et $w(x)$ définie par (\ref{m2}) et (\ref{m11}) et l'action de la dérivation définie par le crochet de Poisson $\{\cdot, \mathbf{H}_{x'}\}$ sur les polynômes $u(x), v(x)$ et $w(x)$ exprimée par les égalités (\ref{eg}), diffèrent par la multiplication de la constante $2$. 
$$\{\mathbf{H}_{x'}, \cdot\}=2\frac{d}{dt}\Big\rvert _{(x', y')}. $$

\bigskip

De prime abord, il n'y a aucun lien qui relie la structure de Poisson introduite par Vanhaneck qui est définie sur la base des racines des polynômes $u(x)$ et $h(x)- v^2(x)$ avec le système dynamique défini par Mumford à l'aide de la dérivée $\frac{d}{dt}$, toutefois leur ajustement est parfait et cela reste énigmatique. De plus certains choix pris par Mumford qui sont nécessaires pour établir le système integrable même si ces choix donnent l'impression d'avoir été faits arbitrairement 
sont qu'il le fait de la chance ou est du génie ? 
\\
Remarquez que le polynôme $h$ en $x$ ne définit pas uniquement l'équation de la courbe hyperelliptique $h(x)=y^2$ mais il définit aussi des Hamiltoniens. Ces Hamiltoniens sont les coefficients de $h$ qui sont polynomiaux en $((u_i)_{0 \leqslant i \leqslant g-1}, (v_i)_{0 \leqslant i \leqslant g-1}), (w_i)_{0 \leqslant i \leqslant g-1})$, et leur crochet de Poisson commutent $(\{h_i, h_j\}=0)_{0 \leqslant i, j \leqslant 2g}$, ces hamiltoniens forment le système intégrable de Mumford. Ce double rôle que joue du polynôme $h$ est si essentiel qu'on peut se demander si c'est une pure coïncidence ou miracle des mathématiques?

\bigskip

\section{Stratification}\label{section3}
Mumford et Vanhaecke ont étudié et décrit la fibre $\mathbf{H}^{-1}(h)$ lorsque le $h$ de $\mathbb{C}_{2g+1}^1[x]$ (resp. $\mathbb{C}_{2g+2}^1[x]$) défini la courbe lisse $\mathcal{C}$ d'équation affine $y^2=h(x)$. Dans cette partie on approfondira nos connaissances en s'intéressant aux cas où $h$ a des racines multiples c'est à dire lorsque la courbe $\mathcal{C}$ d'équation affine $y^2=h(x)$ a des points singuliers. 
À cet égard, on introduira trois stratifications, la première stratification sera basée sur l'étude algébrique de l'ensemble $M_g$, la seconde stratification sera fondée sur une description géométrique détaillée de $M_g$ ainsi que la fibre $\mathbf{H}^{-1}(h)$, la troisième stratification est la plus fine stratification de la fibre $\mathbf{H}^{-1}(h)$ exposée dans cet article et elle combine les spécificités des deux précédentes stratifications. 

\bigskip

Nous rappelons la définition de stratification: 
\begin{defi*}
Soit $(I, \leqslant)$ un ensemble (partiellement) ordonné et soit $V$ une variété algébrique. Une stratification de $V$ est une partition $(S_i)_{i\in I}$ de $V$ où chaque élément $S_i$ est une variété quasi-affine de $V$ telle que sa fermeture de Zariski de vérifie la condition suivante:
$$\overline{S_i}=\displaystyle{\bigsqcup_{j \leqslant i} S_j} \; \text{ pour tout } i\in I. $$
On appelle $S_i$ une strate de la stratification $(S_i)_{i\in I}$. 
\\
Soient deux stratifications $(S_i)_{i\in I} $ et $(S'_j)_{j\in J} $ de la variété algébrique $V$. La stratification $(S'_j)_{j\in J} $ est dite plus fine que $(S_i)_{i\in I} $, si pour tout $j\in J$ il existe un unique $i\in I$ tel que $S'_j \subseteq S_i$. 
\end{defi*}

\subsection{}\label{strati1}Avant de définir la première stratification de l'ensemble $M_g$, nous introduisons la notation suivante: 
Soit $B(x)=\left(\begin{array}{cc}
a (x)&{c(x)} \\
b (x)& -a(x)
\end{array} \right)$ une matrice de $\mathfrak{sl}_2(\mathbb{C}[x])$ on définit le $\PGCD(B(x))= \PGCD(a(x), b(x), c(x))$. 
\begin{defi*}
On définit l'application $\rho$ de la manière suivante:
\begin{equation}
\begin{array}{cccl}
\rho:& \mathfrak{sl}_2(\mathbb{C}[x]) &\longrightarrow &\mathbb{N}\\
& B(x) &\longrightarrow & \deg(\PGCD B(x))
\end{array}
\end{equation}
La restriction de l'application $\rho$ à la sous-variété affine $M_g$ est l'application $\rho_g$: 
\begin{equation*}
\begin{array}{cccl}
\rho_g:&M_g &\longrightarrow &\{0, 1, 2, \cdots, g-1\}\\
& \left(\begin{array}{cc}
v (x)&{u(x)} \\
w (x)& -v(x)
\end{array} \right)&\longrightarrow & \deg(\PGCD(u(x), v(x), w(x)))
\end{array}
\end{equation*}
Il nous arrivera d'omettre d'écrire l'indice $g$ de l'application restriction $\rho_g$ s'il n'y a pas d'ambiguïté. 
\end{defi*}
L'application $\rho_g$ est surjective car le $\deg{v(x)} \leqslant g-1 <\deg{u(x)}<\deg{w(x)}$. L'image inverse par $\rho_g$ de $g-i$ avec $0 \leqslant g-i \leqslant, g-1$ est notée: \begin{equation*}\label{S_i}S_{g, i}=\rho_g^{-1}(g-i). \end{equation*}
Le degré du $\PGCD(u(x), v(x), w(x))$ est unique, alors pour tout deux elements distincts $0 \leqslant i, j \leqslant g-1$, on a $S_{g, i} \cap S_{g, j}=\emptyset$. 

\medskip

Soit $P$ un polynôme de degré $n>0$
$$P(x) = \sum_{i=0}^n c_ix^i. \eqno{(\ast)}$$
Soit $l \in \mathbb{N}^*$ on note par $M:= M_{P, l}$ la matrice $ l\times (n+l)$ dont les entrées \begin{align}\label{Matrx}
 m_{r, s}= c_{n+(r-s)}, \text{ pour $1 \leqslant r \leqslant l$ et $1 \leqslant s \leqslant n+1$}, 
\end{align} avec la convention suivante si $j \notin [0, n]$ on a $c_j= 0$. La matrice $M:= M_{P, l}$ est appelée une matrice Toeplitz. 
\\
Soit $\gamma$ une racine de $P$ d'ordre $k$, alors pour tout $j \in [0, k-1]$ et $m \in \mathbb N$, on a la $j$ dérivée suivante:
$$(x^mP(x))^{(j)}(\gamma) =0. \eqno{(\ast \ast)}$$

Soit $s=n-i$, on note 
$$b(s)_j =[x^{i+m}]^{(j)}(\gamma)=\frac{(n+m-s)!}{(n+m-s-j)!}\gamma^{n+m-s-j}=\frac{(i+m)!}{(i+m-j)!}\gamma^{i+m-j }, $$
Des égalités $(\ast), (\ast\ast)$ et par notre convention sur les coefficients $\{c_j\}$, on obtient que
$$(x^mP(x))^{(j)}(\gamma)=\sum_{i=j-m}^{n} c_{i}\frac {(i+m)!}{(i+m-j)!}\gamma^{i+m-j }=\sum_{s=0}^{n+m-j} c_{n-s}b(s)_j =0, 
\eqno {(\ast\ast\ast)}$$
Soit $j \in [0, k-1]$. On définit le vecteur colonne $v_{j
}$ de dimension $(n+l)\times 1$ , dont les entrées sont 
\begin{align}\label{vg}
v_{j
}^s
=\left\{\begin{array}{cl}b(s)_j &\text{ pour }{0 \leqslant s \leqslant n+m-j}, \\ 
0 &\text{ sinon } \end{array} \right. \end{align}

 %
%
On a \begin{equation} \label{process}M_{P, l}v_j =(\sum_{s=t}^{n+m-j} c_{n-s+t}b(s)_{j})_{0 \leqslant t \leqslant l}= 
[(x^{m-t}P(x))^{(j)}(\gamma)]_{0 \leqslant t \leqslant l}\end{equation} De $(\ast\ast)$, on a que $$M_{P, l}v_j=(0^{l}), $$
c'est à dire $v_j \in \ker (M_{P, l})$.
On remarque que les $(v_{j})_{j \in [0, k-1]}$ sont linéairement indépendants car pour tout $ {j \in [0, k-1]}$ uniquement les $n+m-j+1$ premières entrées du vecteurs $ v_j$ sont non nuls, alors car $v_{j}^{n+m-j}=1
$ et $v_{j+t}^{n+m-j} = 0$ pour ${t>0 
}$, par conséquent le vecteur $ v_j$ et les vecteurs $ (v_{j+t})_{t \in [1, k-j-1]}$ sont linéairement indépendants.
\\
Soit $(\gamma_r)_{0 \leqslant r \leqslant u}$ les racines $P(x)$ de multiplicité $(k_r)_{0 \leqslant r \leqslant u}$
. 
En répétant le procédé incarné par l'équation (\ref{process}) pour chaque racine $(\gamma_r)$ de $P$, on obtient $k_r$ vecteurs $(v_{j_r, \gamma_r})_{0 \leqslant r \leqslant u \, ; \, 0 \leqslant j_r \leqslant k_r-1 }$ dans $\ker M_{P, l}$.\\
Nous allons montrer l'égalité suivante:
\begin{equation}\label{ker1}\ker M_{P, l}=\langle v_{j_r, \gamma_r}\rangle_{0 \leqslant r \leqslant u \, ; \, 0 \leqslant j_r \leqslant k_r-1 }.\end{equation}
La matrice $M_{P, l}$ de $l\times (n+l)$, alors le rang $\rg M_{P, l} \leqslant l$. La sous-matrice $M_{P, l, l}$ de $M_{P, l}$ composée des $l$ premières colonnes est triangulaire supérieure par conséquent elle est de rang $l$, ceci entraine que $\rg M_{P, l} = l$ et $ \dim\ker M_{P, l}= n$, alors
\begin{align}\label{pn}
\dim\ker M_{P, l}= \deg(P).
\end{align}
Par l'égalité (\ref{process}) on a l'inclusion $\langle v_{j_r, \gamma_r}\rangle_{0 \leqslant r \leqslant u \, ; \, 0 \leqslant j_r \leqslant k_r-1 } \subseteq \ker M_{P, l}.$
Pour prouver l'égalité (\ref{ker1}), il suffit de s'assurer que les vecteurs $(v_{j_r, \gamma_r})_{0 \leqslant r \leqslant u \, ; \, 0 \leqslant j_r \leqslant k_r-1 }$ sont linéairement indépendants.
On a vu que $(v_{j_r, \gamma_r})_{ 0 \leqslant j_r \leqslant k_r-1 }$ sont linéairement indépendants. 
Pour $ 0 \leqslant j_r \leqslant \max\{k_r-1 | \; 0 \leqslant r \leqslant u \}$, 
 la matrice de Vandermond de $[x^mP(x)]^{j_r-1}$ nous assure que les vecteurs $(v_{j_r, \gamma_r})_{ 0 \leqslant r \leqslant u }$ sont linéairement indépendants, 
on déduit que 
les $n$ vecteurs $\langle v_{j_r, \gamma_r}\rangle_{0 \leqslant r \leqslant u \, ; \, 0 \leqslant j_r \leqslant k_r-1 }$ sont linéairement indépendants et par conséquent 
$$\ker M_{P, l}=\{v_{j_r, \gamma_r}\; | \;{0 \leqslant r \leqslant u; 0\leqslant j_r \leqslant k_r-1 }\}. \eqno{(i)} $$

\medskip

Nous allons utiliser la matrice Toeplitz pour un polynôme afin de déterminer le degré du $\PGCD$ de trois polynômes. 
\\
Soient les polynômes $P(x)=\sum_{i=0}^n p_ix^i $, $Q(x) = \sum_{i=0}^m q_ix^i$ et $Z(x) = \sum_{i=0}^y z_ix^i$ de degré $n=\deg P \, , m=\deg Q $\, $y=\deg Z $. On note $C=\max(n+m+1, y+1)$. 
Soit la matrice $T(P, Q, Z)$ de dimension $(n+m+y)\times C $ dont les entrées sont 
$$t_{r, s}=\left\{ \begin{array}{l} p_{n+(r-s)} \text{ pour } 1 \leqslant r \leqslant m, \\
 q_{r-s} \text{ pour } m+1 \leqslant r \leqslant n+m\\
 z_{y+(r-s)-m-n} \text{ pour } m+n+1 \leqslant r \leqslant n+m+y. 
\end{array}\right. $$
Avec 
la convention suivante: si $j \notin [0, n], k \notin [0, m]$ et $ l \notin [0, y]$ on a $p_j= 0, q_k= 0 $ et $z_l= 0 $. \\
On note par $(0^{i\times j})$ la matrice nulle de dimension $i\times j$. \\
Soient les matrices $M_{P, m}$, $M_{Q, n}$ et $M_{Z, y}$ définies par la formule (\ref{Matrx}). 
La matrice $T(P, Q, Z)$ peut être exprimée 
comme la 
 composée des bloques $(M_{P, m}, 0^{m \times (C-m-n-1)}) $ se situant au-dessus des bloques de $(M_{Q, n}, 0^{n \times (C-m-n-1)})$ qui sont à leur tour au-dessus des bloques $(M_{Z, y}, 0^{y \times (C-y-1)})$. 

\medskip

Soit le polynôme $D(x)=\PGCD (P(x), Q(x), Z(x))$ de degré $ d$.\\ Les matrices $M_{D, C-d-1}$, 
 $(M_{P, m}, 0^{m \times (C-m-n-1)})$, $(M_{Q, n}, 0^{n \times (C-m-n-1)})$, $(M_{Z, y}, 0^{y \times (C-y-1)})$ ont $C$ colonnes. \\

Montrons que $\displaystyle{\ker(\mathbf{M}_{D, C-d-1})}=\ker(T(P, Q, Z))$. Par définition le noyau de la matrice $T(P, Q, Z)$ est :
$$\ker(T(P, Q, Z))= \ker(M_{P, m}, 0^{m \times (C-m-n-1)})\cap \ker(M_{Q, n}, 0^{n \times (C-m-n-1)})\cap \ker(M_{Z, y}, 0^{y \times (C-y-1)}).$$

Montrons que $\ker(M_{P, m}, 0^{m \times (C-m-n-1)})\cap \ker(M_{Q, n}, 0^{n \times (C-m-n-1)})\cap \ker(M_{Z, y}, 0^{y \times (C-y-1)})$ est non vide si et seulement si $P$, $Q$ et $Z$ ont des racines communes.\\ 
Soit $\gamma$ une racine de $D(x)$ de multiplicité $k$ alors $\gamma$ est une racine de $P$ (resp. $Q$, $Z$) de multiplicité $k_1$ ($k_2$ et $k_3$ respectivement), avec $k=\min(k_1, k_2, k_3)$ alors en vertu de ${(i)}$, (\ref{vg}) et (\ref{process}), on a tout pour $j \in [0, k-1]$
\begin{align*}M_{D, C-d-1}(v_{j, \gamma})&=(M_{P, m}, 0^{m \times (C-m+n)})(v_{j, \gamma})\\
&= (M_{Q, n}, 0^{n \times (C-m-n-1)})(v_{j, \gamma})\\
&= (M_{Z, y}, 0^{y \times (C-m-n-1)})(v_{j, \gamma})\\
&= 0^{ (m+n+y)\times 1}
\end{align*}
On déduit que
%
\begin{align*}\ker M_{D, m+n-d} &= \ker(M_{P, m}, 0^{m \times (C-m-n-1)}) \cap \ker(M_{Q, n}, 0^{n \times (C-m-n-1)}) \cap \ker(M_{Z, y}, 0^{y \times (C-y-1)}), \\ 
&=\ker(T(P, Q, Z))
\end{align*}
%
 Par l'égalité (\ref{pn}) on en conclut que
\begin{equation}\label{noyau}\dim \ker T(P, Q, Z) = \deg \PGCD (P, Q, Z). \end{equation}
De plus comme le $\deg\PGCD (P, Q, Z)=d$, alors $$\text{Pour $ 0 \leqslant l \leqslant d-1$, tous les mineurs de $T(P, Q, Z)$ d'ordre $l$ sont nuls. } \eqno{(ii)}$$

\begin{prop}\label{p30}
 La famille $(S_{g, i})_{i\in\{0, \dots, g\}}$ définit une stratification de l'espace affine $M_{g}$. 
\end{prop}
\begin{proof}[Preuve]
Soit $0 < i \leqslant g$. On montre que l'image inverse de $\{i, \dots, g\}$ par $\rho$ est un fermé de Zariski de
$M_g$. 

\smallskip

L'image inverse $$\rho^{-1}(\{i, \dots, g\})= \bigsqcup\limits_{j=0}^{g-i}S_{g, j}=\left\{ A(x)=\left(\begin{array}{cc} v(x) & u(x) \\ w(x) & -v(x)
\end{array} \right) \in M_{g} \text{ tels que } \deg(\PGCD(u, v, w))\geqslant i\right\}. $$ 
Les entrées de la matrice de Toeplitz $T(u, v, w)$ sont les coefficients de $u, v$ et $w$, alors les mineurs de $T(u, v, w)$ sont polynomiaux en $(u_i)_{0 \leqslant i \leqslant g-1}, (v_i)_{0 \leqslant i \leqslant g-1}$ et $(w_i)_{0 \leqslant i \leqslant g}$. \\
 Si $A(x)=\left(\begin{array}{cc} v(x) & u(x) \\ w(x) & -v(x)
\end{array} \right) \in \rho^{-1}(\{i, \dots, g\})$ implique $\deg(\PGCD(u, v, w))\geqslant i$. D'après $(ii)$, on a 
que les mineurs d'ordre $j$ avec $i+1 \leqslant j \leqslant g $ de $T(u, v, w)$ sont nuls et par conséquent 
 $\rho^{-1}(\{i, \dots, g\})= \bigsqcup\limits_{j=0}^{g-i}S_{g, j}$ est un fermé de Zariski de $M_{g}$. \\ 
On a \begin{align*}
S_{g, g-i}&=\bigsqcup\limits_{j=0}^{g-i}S_{g, j} -\bigsqcup\limits_{j=0}^{g-i-1}S_{g, j}. 
\end{align*}
donc $S_{g, g-i}$ est le complémentaire d'un fermé donc un ouvert de Zariski du fermé $\bigsqcup\limits_{j \leqslant g-i}S_{g, j}$ et donc une variété quasi-affine, par conséquent la fermeture de $S_{g, g-i}$ est $\overline{ S_{g, g-i}}= \bigsqcup\limits_{j \leqslant g-i}S_{g, j}$. 
Par conséquent, les fibres de l'application $\rho$ définissent bien une stratification de $M_g$. 
\end{proof}

\bigskip

\subsection{}\label{strati2}Nous définissons maintenant la deuxième stratification de l'ensemble $M_g$, qui est décrite à travers l'application $\sigma$ définie de la façon suivante:
$$\begin{array}{cccl}
\sigma:& M_g & \longrightarrow & \{1, \dots, g\} \\ 
 & A(x) &\longrightarrow & \dim \left\langle D_0\vert_{A(x)}, \dots, D_{g-1}\vert_{A(x)}\right\rangle. 
\end{array}$$
La fibre de $\sigma$ au dessus de $i$ est notée par $M_{g, i}$:
\begin{align} \label{Mi}
 M_{g, i}&=\{A(x)\in M_{g}\mid \dim \left\langle D_0\vert_{A(x)}, \dots, D_{g-1}\vert_{A(x)}\right\rangle=i\}. 
\end{align}
On note par $I$ l'ensemble des entiers $1 \leqslant i \leqslant g$ tels que $M_{g, i} \neq \emptyset $. On verra plus loin (proposition \ref{prop1bis}) que $I=\{1, \cdots , g\}$.
\begin{prop}\label{st}
La famille $(M_{g, i})_{i\in I}$ définit une stratification de $M_g$. 
\end{prop}
\begin{proof}[Preuve]
Soit $A(x)$ une matrice de $M_g$. Le rang des vecteurs polynomiaux $ \dim \left\langle D_0\vert_{A(x)}, \dots, D_{g-1}\vert_{A(x)}\right\rangle$ est unique, par conséquent pour tout deux éléments distincts $ i, j $ de $I$, on a $M_{g, i} \cap M_{g, j}=\emptyset $. De plus comme l'application surjective $\sigma$, alors $M_{g}=\bigsqcup\limits_{i\in I}M_{g, i}$. \\
Nous montrons que pour tout ${i\in I}$ le sous ensemble $
\bigsqcup\limits_{j \leqslant i } M_{g, j}
$ est constitué de fermés de Zariski de l'espace affine $M_{g}$. 
 D'après l'équation (\ref{c'}) on sait que les champs de vecteurs $D_{k}$ sont des fonctions polynomiales en ${u_{0}, \dots, u_{g-1}}$, ${v_{0}, \dots, v_{g-1}}$, ${w_{0}, \dots, w_{g}}$. Soit $\mathcal{D}=(D_{0}, \dots, D_{g-1})$ une matrice $(3g+1)\times g$ à coefficients dans
 D'après l'égalité (\ref{Mi}), 
$$\bigsqcup\limits_{j \leqslant i } M_{g, j}=\{A(x)\in M_{g} \mid \dim \langle D_{0}\vert_{A(x)}, \dots, 
D_{g-1}\vert_{A(x)} \rangle \leqslant i \}. $$ Autrement dit une matrice $A(x)$ appartient à $\bigsqcup\limits_{j \leqslant i } M_{g, j}$ si et seulement si 
tous les mineurs d'ordre $k > i $ de la matrice $\mathcal{D}$ sont
nuls. D'après l'équation (\ref{c'}), les vecteur $ D_0\vert_{A(x)}, \dots, D_{g-1}\vert_{A(x)}$ sont vecteurs polynomiaux en ${(u_{i})_{0 \leqslant i \leqslant g-1}}$, ${(v_{i})_{0 \leqslant i \leqslant g-1}}$, ${(w_{i})_{0 \leqslant i \leqslant g+1}}$. Alors, les mineurs de la matrice $\mathcal{D}=((D_i)_{0 \leqslant i \leqslant g-1}) $ sont aussi des polynômes en ${(u_{i})_{0 \leqslant i \leqslant g-1}}$, ${(v_{i})_{0 \leqslant i \leqslant g-1}}$, ${(w_{i})_{0 \leqslant i \leqslant g+1}}$.
 On conclut que l'ensemble 
 $\bigsqcup\limits_{j \leqslant i } M_{g, j}$ est un fermé de Zariski de
$M_{g}$. 
L'ensemble $M_{g, i}$ est une variété quasi-affine car $M_{g, i}$ est le complémentaire du fermé de Zariski 
:
$$M_{g, i}=\bigsqcup\limits_{j \leqslant i } M_{g, j}-\bigsqcup\limits_{j \leqslant i-1} M_{g, j}. $$
Par conséquent, $M_{g, i}$ est un ouvert dense de Zariski de $ \bigsqcup\limits_{j \leqslant i } M_{g, j}$ 
La famille $(M_{g, i})_{i\in I}$ est bien une stratification de $M_{g}$. 
\end{proof}
\bigbreak

\subsection{}\label{strati12} Le théorème suivant est important car il associe l'aspect géométrique de l'ensemble $M_g$ et l'aspect algébrique en chaque point $A(x)$ de $M_g$, il est aussi primordial car il nous permettra d'établir le lien entre les deux stratifications $(S_{g, i})_{0 \leqslant i \leqslant g-1}$ et $(M_{g, i})_{i\in I}$. 
\begin{theo} \label{prop1}
Les champs de vecteurs $D_0, \dots, D_{g-1}$ sont linéairement indépendants au point $A^0(x)=\left(\begin{array}{cc}
v^0(x)& u^0(x)\\ 
w^0(x)& -v^0(x)
\end{array} \right)$ de ${M}_g$ si et seulement si le $\deg\PGCD(A^0(x))=0$. 
\end{theo}
\begin{proof}[Preuve]
Soit $A^0(x) $ un point de $ {M}_g $ où les champs de vecteurs $D_0, \dots, D_{g-1}$ sont linéairement dépendants, alors il existe 
des constantes non toutes nulles $(a_i)_{0 \leqslant i \leqslant g-1}\in\mathbb{C}^g/ \{0\}$ telles que
$$ \sum\limits_{i=0}^{g-1} a_iD_i\vert_{A^0(x)} = 0. $$
En évaluant l'équations (\ref{c'}) au point $A^0(x)$, on a:
\begin{align}
 \sum\limits_{i=0}^{g-1} a_i D_i\vert_{A^0(x)}&= \sum\limits_{i=0}^{g-1} a_i \left[A^0(x), \left[ \displaystyle{\frac{A^0(x)}{x^{i+1}}}\right]_+ - \left(\begin{array}{cc}
0 & 0 \\ 
{u^0(x)_i} & 0
\end{array} \right)\right], \label{d'}
\end{align}
en développant l'égalité précédente (\ref{d'}), on obtient:
\begin{align}\label{Q, R}
\sum\limits_{i=0}^{g-1} a_i D_i\vert_{A^0(x)} u(x)&=2 v^0(x)\underbrace{ \sum\limits_{i=0}^{g-1} a_i
 \left[\frac{u^0(x)}{x^{i+1}}\right] _+}_{Q(x)}-2u^0(x)\underbrace{\sum\limits_{i=0}^{g-1} a_i
 \left[\frac{v^0(x)}{x^{i+1}}\right]_+ }_{R(x)}=0, \\ \label{S, Q}
 \sum\limits_{i=0}^{g-1} a_i D_i\vert_{A^0(x)} v(x) &=
u^0(x)\underbrace{ \sum\limits_{i=0}^{g-1} a_i (\left[\frac{w^0(x) }{x^{i+1}}\right]_+-
 u^0(x)_{i})}_{S(x)}-w^0(x)\underbrace{ \sum\limits_{i=0}^{g-1} a_i \left[\frac{u^0(x)}{x^{i+1}}\right]_+ }_{Q(x)}=0, 
\\ \nonumber
 \sum\limits_{i=0}^{g-1} a_i D_i \vert_{A^0(x)}w(x)&= 2w^0(x)\underbrace{ \sum\limits_{i=0}^{g-1} a_i
 \left[\frac{v^0(x)}{x^{i+1}}\right]_+ }_{R(x)}-2v^0(x)\underbrace{ \sum\limits_{i=0}^{g-1} a_i
 (\left[\frac{w^0(x)}{x^{i+1}}\right]_+ - u^0(x)_{i})}_{S(x)}=0. 
\end{align}
Les égalités (\ref{Q, R}) et (\ref{S, Q}) nous définissent trois polynômes $Q(x), R(x)$ et $S(x)$ en $x$. 
On sait que le polynôme $u^0(x)$ est unitaire de degré $
g$ et le polynôme $w^0(x)$ est unitaire de degré $
g+1, $ et par hypothèse le $g$-uplet $(a_i)_{0 \leqslant i \leqslant g-1}$ est différent de zéro et par définition, alors l'égalité (\ref{S, Q}) entraine que $\deg(S(x))>0$, et ceci implique que $Q(x)$ (resp. $S(x)$) est non nul de degré au plus $g-1$ (resp. $g$). Toujours par l'égalité $u^0(x)S(x)=w^0(x)Q(x)$ on a que les racines ainsi que leurs multiplicités des
 polynômes $ u^0(x)S(x) $ et $w^0(x)Q(x)$ sont identiques, et comme le degré $\deg Q(x) \leqslant g-1 $ et $S(x) \leqslant g $, il en résulte qu'il existe un facteur en commun entre les
polynômes $u^{0}(x)$ et $w^{0}(x)$ de degré au moins $1$. Plus précisément, il existe $a \in \mathbb{C}$ tel que 
 $$u^0(x)=(x-a)^{\mu_u} {u^{0}}'(x), \text { où $\mu_u$ est la multiplicité de la racine, }$$ 
 et $a$ est aussi une racine de $w^0(x)Q(x)$ avec
 $$Q(x)=(x-a)^{\mu_Q} P(x) \text { où $\mu_Q$ est la multiplicité de la racine, 
 avec } \mu_Q< \mu_u.$$
Alors $a$ est une racine de $w^0(x)$ de multiplicité $\mu_u-\mu_Q$. \\
 Par l'égalité (\ref{Q, R}), $v^0(x)Q(x)=u^0(x)R(x) $, 
et d'après ce qui précèdent $$v(x)(x-a)^{\mu_Q} P(x)=(x-a)^{\mu_u} {u^{0}}'(x)R(x), $$
à l'aide du même argument $a$ est une racine de $v^0(x)$ de multiplicité $\mu_u-\mu_Q$. 
Par conséquent, $u^0(x), v^0(x)$ et $w^0(x)$ ont au moins une racine commune. On conclut que si $D_0, \dots, D_{g-1}$ sont
linéairement dépendantes au point $A^{0}$, alors le $\deg \PGCD(A^0(x)) \neq 0$.

On montre maintenant la réciproque, soit $A^0(x)$ est une matrice de $\deg\PGCD(A^0(x)) \neq 0$, alors il existe un $a \in \mathbb{C}$ tel que 
$A^0(a)=0$. Rappelons pour tout $z \in \mathbb{C}$ on a
$$
 D_z\vert_{A^0(x)}=\sum \limits_{i=0}^{g-1} z^{i} D_{i}\vert_{A^0(x)} = \left[{A^0(x)}, -\displaystyle{\frac{A^0(z)}{x-z}}- \left(\begin{array}{cc}
0 & 0 \\ 
u^0(z) & 0
\end{array} \right)\right]\;. 
$$
En évaluant au point $z=a$ on obtient
$$
\sum \limits_{i=0}^{g-1} a^{i} D_{i}\vert_{A^0(x)}= \left[{A^0(x)}, \left(\begin{array}{cc}
0 & 0 \\ 
0& 0
\end{array} \right)\right]= 0. 
$$
Par conséquent, les $g$ champs de vecteurs $D_0, \dots, D_{g-1}$ sont linéairement dépendants\footnote{Si $a=0$ on a $D_{0}=0$ et par $ \sum \limits_{i=0}^{g-1} a_{i} D_{i}= 0$ pour tout $a_0\in \mathbb{C}$ et $(a_i=0)_{ 1\leqslant i \leqslant g-1}$. } au point $A^0(x)$. 
\end{proof}

\bigskip

Soit $P(x)$ un polynôme unitaire de $\mathbb{C}[x]$ de degré $n \in \mathbb{N^{\ast }} $. On note par $\mu_{P}$
l'application affine définie de la manière suivante:
$$\begin{array}{cccl}
\mu_{P}: & M_g & \longrightarrow & M_{g+n} \\ 
 & A(x) & \longrightarrow & P(x)A(x). 
\end{array} $$
la différentielle de $\mu_{P}$ est 
$$\begin{array}{cccl}
d \mu_{P}: & TM_g & \longrightarrow &TM_{g+n} \\ 
 & D & \longrightarrow & d \mu_{P}(D), 
\end{array} $$
L'application $\mu_{P}$ est un isomorphisme affine sur son image. Sa différentielle $d \mu_{P}$ 
est notée:
$$ {\mu_{\ast}}_P(D\vert_{A^0(x)}) = d \mu_{P}(D\vert_{(A^0(x))}). $$

\begin{prop}\label{lem1} 
Soit $P(x)$ un polynôme unitaire de $\mathbb{C}[x]$ de degré $n\in \mathbb{N}^{*}$, et soient $A^0(x) \in M_g$ et $y\in \mathbb{C}$. On a
\begin{align}\label{+}
 D^{M_{g+n}}_y\vert_{{\mu_{P}}(A^0(x))}&=P(y){\mu_{\ast}}_P (D^{M_g}_y\vert_{A^0(x)}). 
\end{align}

\end{prop}
\begin{proof}[Preuve]
Soit une matrice $A^0(x)=\left(\begin{array}{cc}
v^0(x) & u^0(x) \\ 
w^0(x) & -v^0(x)
\end{array} \right)\in M_g$, et soit 
le champ de vecteurs $D_{y}^g$ au point $A^0(x)$ défini par (\ref{cc}). L'image de $D_{y}^g$ par l'application linéaire 
 $\mu_{\ast P}$ est:
\begin{align}
\mu_{\ast P}(D^{M_g}_y\vert_{A^0(x)})&= P(x)\left[{A^0(x)}, -\displaystyle{\frac{A^0(y)}{x-y}}- \left(\begin{array}{cc}
0 & 0 \\ 
u^0(y) & 0
\end{array} \right)\right]. \label{1}
\end{align}
Le champs de vecteurs $D^{M_{g+n}}_y$ au point ${\mu_{P}}(A^0(x))$ de $M_{g+n}$ est:
\begin{align}\nonumber
 D^{M_{g+n}}_y\vert_{{\mu_{P}}(A^0(x))}&= \left[{P(x)A^0(x)}(x), -\displaystyle{\frac{P(y)A^0(y)}{x-y}}- \left(\begin{array}{cc}
0 & 0 \\ 
P(y)u^0(y) & 0
\end{array} \right)\right], \\
&= P(x)P(y)\left[{A^0(x)}, -\displaystyle{\frac{A^0(y)}{x-y}}- \left(\begin{array}{cc}
0 & 0 \\ 
u^0(y) & 0
\end{array} \right)\right]. \label{2}
\end{align}
En remplaçant (\ref{1}) dans (\ref{2}) on a:
$$
D^{M_{g+n}}_y\vert_{{\mu_{P}}(A^0(x))}=P(y)\mu_{\ast P}(D^{M_g}_y\vert_{A^0(x)}). 
$$
\end{proof}
\begin{defi*}
Soit $Q$ un polynôme. On note par $\Res_{y=0} \frac{Q(y)}{y^i}$ le résidu du quotient $\frac{Q(y)}{y^i}$, défini de la manière suivante:
$$\Res_{y=0} \frac{Q(y)}{y^{i+1}}= \left[ \frac{Q(y)}{y^{i}}\right]_+- y\left[ \frac{Q(y)}{y^{i+1}}\right]_+. $$
\end{defi*}
\begin{coro}\label{coco}
Pour tout $0 \leqslant i \leqslant g+n-1$ on a:
$$D^{M_{g+n}}_i\vert_{{\mu_{P}}(A^0(x))}= 
\Res_{y=0} \frac{P(y)\mu_{\ast P}(D^{M_g}_y\vert_{A^0(x)})}{y^{i+1}}. $$
\end{coro}
\begin{proof}[Preuve]
Soit $y\in \mathbb{C}$. Par définition 
$$D^{M_{g+n}}_y\vert_{{\mu_{P}}(A^0(x))}= P(y)\mu_{\ast P}(D^{M_g}_y\vert_{A^0(x)})=\sum\limits_{i=0}^{g+n-1}y^i D^{M_{g+n}}_i\vert_{{\mu_{P}}(A^0(x))}, $$
alors 
\begin{align*}
 \Res_{y=0} \frac{P(y)\mu_{\ast P}(D^{M_g}_y\vert_{A^0(x)})}{y^i}&= \left[ \frac{D^{M_{g+n}}_y\vert_{{\mu_{P}}(A^0(x))}}{y^{i}}\right]_+- y\left[ \frac{D^{M_{g+n}}_y\vert_{{\mu_{P}}(A^0(x))}}{y^{i+1}}\right]_+, \\
&=\sum\limits_{j=i}^{g+n-1} y^{j-i} D^{M_{g+n}}_j\vert_{{\mu_{P}}(A^0(x))} -y \sum\limits_{j=i+1}^{g+n-1}y^{j-i-1} D^{M_{g+n}}_j\vert_{{\mu_{P}}(A^0(x))}, \\
&=D^{M_{g+n}}_i\vert_{{\mu_{P}}(A^0(x))}. 
\end{align*}

\end{proof}

\bigskip

Soit un polynôme $P\in \mathbb{C}^1[x]$ tel que $P(x)=\sum\limits_{i=0}^na_ix^i$, et soit $A^0(x) \in M_g$. 
Par le corollaire \ref{coco}, on a pour tout $0 \leqslant k \leqslant g+n-1$:
\begin{align}\nonumber 
D^{M_{g+n}}_{k}\vert_{{\mu_{P}}(A^0(x))}&= \Res_{y=0} \frac{\sum\limits_{i=0}^na_iy^i\mu_{\ast P}(D^{M_g}_y\vert_{A^0(x)})}{y^{k+1}}, \\ \nonumber
&= \left[ \frac{\sum\limits_{i=0}^na_iy^i\mu_{\ast P}(D^{M_g}_y\vert_{A^0(x)})}{y^{k}}\right]_+- y\left[ \frac{\sum\limits_{i=0}^na_iy^i\mu_{\ast P}(D^{M_g}_y\vert_{A^0(x)})}{y^{k+1}}\right]_+, \\ \nonumber
&=\left[ \frac{\sum\limits_{i=0}^n a_i y^i\sum\limits_{i=0}^n a_ix^i\sum\limits_{i=0}^{g-1} y^i D^{M_g}_i\vert_{A^0(x)}}{y^{k}}\right]_+- y\left[\frac{\sum\limits_{i=0}^na_iy^i\sum\limits_{i=0}^na_ix^i(\sum\limits_{i=0}^{g-1}D^{M_g}_i\vert_{A^0(x)})}{y^{k+1}}\right]_+, \\ \nonumber
&=\left[ \frac{\sum\limits_{i=0}^n a_ix^i\sum\limits_{i=0}^{g+n-1}y^i \sum\limits_{l+m=i} a_lD^{M_g}_m\vert_{A^0(x)}}{y^{k}}\right]_+- y\left[ \frac{\sum\limits_{i=0}^n a_ix^i\sum\limits_{i=0}^{g+n-1}y^i \sum\limits_{l+m=i} a_lD^{M_g}_m\vert_{A^0(x)}}{y^{k+1}}\right]_+, \\ \nonumber
&=\sum\limits_{i=0}^n a_ix^i \sum\limits_{l+m=k} a_lD^{M_g}_m\vert_{A^0(x)}, \\ \label{alg}
D^{M_{g+n}}_{k}\vert_{{\mu_{P}(A^0(x))}}&=\mu_{\ast P}\sum\limits_{l+m=k} a_l D^{M_g}_m\vert_{A^0(x)}, 
\end{align}
nous utiliserons l'égalité (\ref{alg}) dans la proposition \ref{prop1bis}.

\bigskip

On a vu précédemment que les applications $\rho$ et $\sigma$ définissent sur la variété $M_g$ deux
stratifications et la proposition suivante va nous permettre d'identifier la stratification algébrique $(S_{g, i})_{i \in\{0, ..., g-1\}}$ et la stratification algébrique $(M_{g, i})_{i\in I}$. 
\begin{prop}\label{prop1bis}
Soit $A^0(x)$ une matrice de $ M_g$ on a
$$\sigma(A^0(x)) = g-\rho(A^0(x)). $$

\end{prop}
\begin{proof}[Preuve]
Soit $A^0(x)$ une matrice de $ M_g$, dont $ \PGCD(A^0(x))$ est le polynôme unitaire $P(x)$
de degré $\rho(A^0(x))=n $. Il existe une unique matrice $A^1(x)$ de $ M_{g-n}$ telle que:
$$A^0(x)=P(x)A^1(x), $$
et le $ \deg \PGCD(A^1(x))=0$. 
Nous allons montrer que $\dim \left\langle D^{M_g}_0\vert_{A^0(x)}, \dots, D^{M_g}_{g-1}\vert_{A^0(x) }\right\rangle = g-\rho(A^0(x))$. \\
Pour tout $y\in \mathbb{C}$, on sait que
\begin{equation}\label{dmg}D^{M_g}_y\vert_{{\mu_{P}}A^1(x)}= \sum\limits_{i=0}^{g-1}y^iD^{M_g}_i\vert_{{\mu_{P}}A^1(x)} \; \text{ et } \; D^{M_{g-n}}_y\vert_{A^1(x)}=
\sum\limits_{i=0}^{g-n-1}y^iD^{M_{g-n}}_i\vert_{A^1(x)}. \end{equation}
En remplacent les deux égalités précédentes (\ref{dmg}) dans l'égalité (\ref{+}) de la proposition \ref{lem1} on obtient 
\begin{equation}\label{l. e}
 \sum\limits_{i=0}^{g-1}y^iD^{M_g}_i\vert_{{\mu_{P}}A^1(x)}= P(y)\sum\limits_{i=0}^{g-n-1}y^i\mu_{\ast P}D^{M_{g-n}}_i\vert_{A^1(x)}
\;. 
\end{equation}
%
L'égalité (\ref{alg}) implique que pour tout $0 \leqslant i \leqslant g$ les champs de vecteurs $D^{M_g}_i$ s'écrit en fonction de 
 $ (\mu_{P
 \ast} D^{M_{g-n}})_{0 \leqslant i \leqslant g-1}$, 
 alors 
\begin{equation*}\label{*} 
 \left\langle D^{M_g}_0\vert_{{\mu_{P}}A^1(x)}, \dots, D^{M_g}_{g-1}\vert_{{\mu_{P}}A^1(x)}\right\rangle =\left\langle \mu_{P
 \ast}D^{M_{g-n}}_0\vert_{A^1(x)}, \dots, \mu_{\ast P}D^{M_{g-n}}_{g-n-1}\vert_{A^1(x) }\right\rangle, \; 
\end{equation*}
par conséquent 
\begin{align*}
\dim\left\langle D^{M_g}_0\vert_{A^0(x)}, \dots, D^{M_g}_{g-1}\vert_{A^0(x)}\right\rangle =\dim\left\langle
\mu_{\ast P}D^{M_{g-n}}_0\vert_{A^1(x)}, \dots, \mu_{\ast P}D^{M_{g-n}}_{g-n-1}\vert_{A^1(x)}\right\rangle. 
\end{align*}
La théorème \ref{prop1} nous assure que les vecteurs $
D^{M_{g-n}}_0\vert_{A^1(x) }, \dots, D^{M_{g-n}}_{g-n-1}\vert_{A^1(x)}$ sont linéairement indépendants car le $\PGCD$ de $A^1(x)$ est de degré nul et du fait que l'application 
$\mu_P $ est injective alors 
l'action de $\mu_P $ sur une famille de champ de vecteurs préserve leur degré l'indépendance, on obtient ainsi que: 
 $$\dim \left\langle \mu_{\ast P}D^{M_{g-n}}_0\vert_{A^1(x) }, \dots, \mu_{\ast P}D^{M_{g-n}}_{g-n-1}\vert_{A^1(x) }\right\rangle=\sigma(A^0(x)) =
g-n. $$ On a donc
\begin{align*}
\sigma(A^0(x)) &= g-\rho(A^0(x)). 
\end{align*}

\end{proof}
La proposition \ref{prop1bis} implique que les deux stratifications $(S_{g, i})_{i \in\{0, ..., g-1\}}$ et $(M_{g, i})_{i \in I}$ de $M_g$ sont identiques $M_{g, i}=S_{g, i}
\neq \emptyset$ pour $0 \leqslant i \leqslant g-1$ 
. Ceci garantit que $I=\{0, \dots, g-1\}$. 
 
\bigskip

\begin{coro}\label{aze}
 Soit $P$ un polynôme unitaire de degré $n$. 
$ \text{ Pour tout } 0 \leqslant i \leqslant g$
:
\begin{align}\label{pa30}
 \mu_{P}(M_{g, i})&=M_{g+n, i} \cap \mu_{P}(M_{g}). 
\end{align}

\end{coro}
\begin{proof}[Preuve]
Soit $0 \leqslant i \leqslant g $. Pour toute matrice $A(x)\in M_{g, i}$ on a 
$\rho(A(x))=g-i$, la proposition \ref{prop1bis} implique que $\rho(P(x)A(x))=g+n
-i$, et entraine que $\mu_{P}(M_{g, i}) \subset M_{g+n, i} $, par conséquent 
\begin{equation}\label{inc1} \mu_{P}(M_{g, i})\subset M_{g+n, i} \cap \mu_{P}(M_{g}). \end{equation}
La variété quasi-affine $M_{g+n, i} \cap
\mu_{P}(M_{g})$ est composée par les matrices $B(x)$ de $M_{g+n}$ telles que $\PGCD(B(x))$ est un
multiple de $P(x)$ avec $\frac{B(x)}{P(x)}\in M_g$ et $\rho(\frac{B(x)}{P(x)})=g-i$ autrement dit $\frac{B(x)}{P(x)}\in M_{g, i}$, alors \begin{equation}\label{inc2}M_{g+n, i} \cap \mu_{P}(M_{g}) \subset \mu_{P}(M_{g, i}).\end{equation} 
Avec les deux inclusions (\ref{inc1}) et (\ref{inc2}) on obtient l'égalité voulue 
$\mu_{P}(M_{g, i})=M_{g+n, i} \cap \mu_{P}(M_{g}). $ 
\end{proof}

%
%
\bigskip

Rappelons que l'application $\mathbf{H}$ définie par (\ref{H}) est surjective, alors 
pour tout polynôme $h\in \mathbb{C}^1_{2g+1}[x]$, on note par $M_{g}(h)$ la fibre au dessus de $h$
. \\La fibre $M_{g}(h)= \mathbf{H}^{-1}(h)$ est une variété affine, car par définition $\mathbf{H}$ est
une application régulière.
\section{La stratification des fibres $M_{g}(h)$}\label{section4}
Dans cette section nous allons décrire la restriction des deux stratifications de la variété affine $M_{g}(h)$. La première stratification est héritée de la stratification de $M_{g}$ où les strates sont déterminées par $M_{g, i}\cap M_{g}(h)$, 
la seconde stratification sera une stratification plus fine qui combinera les caractéristiques géométriques et algébriques des deux stratifications définies à l'aide des application $\sigma$ et $\rho$ sur des fibres $M_{g}(h)$.
\subsection{}\label{intro-strati3}Introduisons quelques définitions. 
\begin{defi*}
Soit $h$ un polynôme de $\mathbb{C}^1_{2g+1}[x]$. Un polynôme unitaire $Q(x)$ de $\mathbb{C}^1[x]$ est appelé 
{un diviseur quadratique de} $h(x)$ si $Q^{2}$ divise $h(x)$. \\ 
On note par $\mathbb{C}[x]_{h}$ l'ensemble des diviseurs quadratiques de $h$ et on note par $\mathbb{C}[x]_{i, h}$ l'ensemble des diviseurs quadratiques de $h(x)$ de degré $i$. \\
On appelle 
 {le degré de non-régularité} de $h(x)$, le degré maximal des diviseurs quadratiques de $h$, noté $\Upsilon (h)$. \end{defi*}
 Observons que $\mathbb{C}[x]_{\Upsilon(h), h}$ est constitué d'un seul polynôme qu'on appellera le 
 {diviseur quadratique maximal de} $h$. On peut exprimer $\mathbb{C}[x]_{i, h}$ comme l'union disjointe suivante: 
 $$\mathbb{C}[x]_{h}= \bigsqcup\limits_{i=0}^{\Upsilon(h)}\mathbb{C}[x]_{i, h}. $$

\begin{lem}\label{X}
Soit $h$ un polynôme de $\mathbb{C}^1_{2g+1}[x]$. Pour toute matrice $A(x)\in M_{g}(h)$ on a: 
$$\rho(A(x)) \leqslant \Upsilon(h). $$
\end{lem}
\begin{proof}[Preuve]
Soit $A(x)$ une matrice de $M_{g}(h)$ avec $P(x)=\PGCD(A(x))$. Alors il existe une matrice
 $A^{1}(x)$ telle que $A(x)=P(x) A^{1}(x)$. Par conséquent, on a $h(x)=-\det(A(x))=-P^{2}(x) \det(A^{1}(x))$, alors le polynôme unitaire $P^2(x)$ divise $h(x)$ donc $P(x)\in \mathbb{C}[x]_{h}$, ce qui implique que $\rho(A(x))=\deg(P(x)) \leqslant \Upsilon(h)$. 
\end{proof}
Rappelons que pour tout ${0 \leqslant i \leqslant g}$, les matrices $A(x)$ composant les strates $S_{g, i} $ sont caractérisées par le degré du $\deg\PGCD(A(x))$, comme les strates de $M_{g, i}=S_{g, i}$ alors les matrices constituants $M_{g, i}$ sont aussi caractérisées par leur $\PGCD$. \\

Soit $h$ un polynôme de $\mathbb{C}^1_{2g+1}[x]$. On note
\begin{equation}\label{Okk}
M_{g, i}(h)=M_{g}(h) \cap M_{g, i}, \, \; \text{ pour } 0 \leqslant i \leqslant g
\end{equation} 
Les ensembles $\{M_{g, i}(h)\}_{0 \leqslant i \leqslant g}$ sont des variétés quasi-affines, car ils sont l'intersection d'une variété affine avec une variété quasi-affine. 
\begin{prop}\label{p@}
Soit $h \in \mathbb{C}^1_{2g+1}[x]$. L'ensemble $M_{g, i}(h)$ est non-vide si et seulement si $g- \Upsilon(h) \leqslant i \leqslant g$, et la fibre $M_{g}(h)$ est stratifiée par $\Upsilon(h)+1$ strates. 
\end{prop}
\begin{proof}[Preuve]
Soit $g- \Upsilon(h) \leqslant i \leqslant g$. Montrons que $M_{g, i}(h) \neq \emptyset$. Pour cela, nous allons construire une matrice $A(x)=\left(\begin{array}{cc}
v(x) & u(x) \\ 
w(x)& -v(x)
\end{array} \right) \in M_{g}(h)$ telle que $\rho(A(x))=g-i$. Rappelons que $\det(A(x))=h(x)$. \\
Soit $Q(x)=\prod\limits_{j=1}^{g-i}(x-\alpha_{j})$ un polynôme de $\mathbb{C}[x]_{g-i, h}$
, tel que $h(x)=Q^2(x)h'(x)$ avec $h'\in \mathbb{C}^1_{i}[x]$. \\
Choisissons $(a_{j})_{1 \leqslant j \leqslant i}$ une famille de $i$ éléments distincts de $ \mathbb{C}$ telle que $h(a_{j})\neq 0$ pour tout $1 \leqslant j \leqslant i$. On définit le polynôme $u(x) \in \mathbb{C}^1_g[x]$ comme il suit 
\begin{align*}
 u(x)&=Q(x)\underbrace{\prod\limits_{j=1}^{i}(x-a_{j})}_{u'(x)}. \\
\end{align*}
Le polynôme $v(x)$ est défini par les équations (\ref{polyv}). 
On a alors
\begin{align*}
(v(\alpha_j)-\sqrt{h(\alpha_j)})=0 \;\;\; 
\, \text{ pour tout }1 \leqslant j \leqslant g- i. 
\end{align*}
Pour tout $1 \leqslant j \leqslant g- i$, $\alpha_j$ est une racine de $h(x)$ donc une racine de $v(x)$ car $v(\alpha_j)=\sqrt{h(a_j)}=0$. 
Si un élément 
$\alpha_\ell$ est répété $k$ dans $(\alpha_{j})_{1 \leqslant j \leqslant g-i}$, 
alors en vertu des équations (\ref{polyv}) on a
\begin{align*}
\frac{d^m}{dt^m}(v(x)-\sqrt{h(x)}) \Big\rvert_{x=\alpha_\ell}=0 \;\;\; 
\, \text{ pour tout }0 \leqslant m \leqslant k. 
\end{align*}
comme $a_{\ell}$ est une racine de $h(x)$ de multiplicité au moins $k$, alors $\frac{d^m}{dt^m}h(x)\Big\rvert_{x=\alpha_\ell}=0$ et par conséquent $\frac{d^m}{dt^m}v(x)\Big\rvert_{x=\alpha_\ell}=0$, ainsi on a 
 \begin{align*}
 v(x)&=\prod_{ j=1}^{g-i}(x-\alpha_j)v'(x), \\
 v(x)&=Q(x)v'(x), 
 \end{align*}
où $v'(x)$ est un polynôme de $\mathbb{C}_{i-1}[x]$ tel que:
\begin{align*}
\begin{array}{l}
(Q(a_j)v'(a_j)-Q(a_j)\sqrt{h'(a_j)})=0, \\
(v'(a_j)-\sqrt{h'(a_j)})=0, 
\end{array}
\end{align*}
on a 
\begin{align*}
 v'(a_j)=\sqrt{h'(a_j)} \neq 0.
\end{align*}
Par définition, le polynôme unitaire $w(x)$ de degré $g+1$ est le quotient suivant:
\begin{align*}
w(x)&=\frac{h(x)-v^{2}(x)}{u(x)}=Q(x)\underbrace{\frac{h'(x)-v'^{2}(x)}{u'(x)}}_{w'(x)}. 
\end{align*}
Comme les racines $(a_{j})_{1 \leqslant j \leqslant i}$ du polynôme $u'(x)$ n'annule pas le polynôme $v'(x)$, alors $\PGCD(u'(x), v'(x))=1$, ceci implique que $\PGCD(u'(x), v'(x), w'(x))=1$, donc
\begin{eqnarray*}
 \PGCD(A(x))&=&\PGCD(Q(x)u'(x), Q(x)v'(x), Q(x)w'(x))\\
 &=&Q(x)\PGCD(u'(x), v'(x), w'(x))=Q(x)\;, 
\end{eqnarray*}
ainsi $\rho(A(x))=\deg(Q(x))=g-i$ de ce fait $A(x)\in M_{g, i}(h)$, d'où l'ensemble $M_{g, i}(h) $ est non-vide
. 

\bigskip

Nous montrerons que pour $i<g-\Upsilon(h)$, alors $M_{g, i}(h)=\emptyset$. 
Supposons que $M_{g, i}(h)$ est non vide alors il existe une matrice $A(x)$ telle que $ i=g-\rho(A(x))$ alors 
\begin{align*}
g- \rho(A(x)) & \leqslant g-\Upsilon(h), \\
\rho(A(x)) &\geqslant \Upsilon(h). 
\end{align*}
 ceci est impossible car c'est en contradiction avec le lemme (\ref{X}) qui nous assure que $\rho(A(x)) \leqslant \Upsilon(h)$ pour tout $A(x) \in M_{g}(h)$. Par conséquent $M_{g, i}(h)=\emptyset$ pour $i<g-\Upsilon(h)$, alors on a $\Upsilon(h) + 1$ strates.

\end{proof}

\begin{prop}\label{p39}
Soit $h\in \mathbb{C}^1_{2g+1}[x]$. La famille $(M_{g, i}(h))_{g-\Upsilon(h) \leqslant i \leqslant g}$ est une stratification de la fibre $M_{g}(h)$. 
\end{prop}
\begin{proof}[Preuve]
La variété affine $M_{g}(h)$ est un fermé de $M_{g}$. donc toute stratification de $M_{g}$ induit une
stratification de $M_{g}(h)$. 
Notons que les champs de vecteurs $(D_i)_{0 \leqslant i \leqslant g-1}$ sont stables sur la fibre $M_{g}(h)$ car pour tout $A(x)\in M_g(h)$ et $0 \leqslant i \leqslant g-1$ on a 
\begin{align*}
D_i\Big\rvert_{A(y)}(h(x))&=\{h_i\Big\rvert_{A(y)}, h(x)\}, \\
&=\{h_i\Big\rvert_{A(y)}, \sum\limits_{j=0}^{2g+1}h_jx^j\}, \\
&=\sum\limits_{j=0}^{2g+1}x^j\{h_i\Big\rvert_{A(y)}, h_j\}
\end{align*}
Comme $(M_g, \{\cdot, \cdot\}, \mathbf{H})$ est un système intégrable alors la famille $(h_0, \dots, h_{2g+1})$ est involutive donc 
\begin{align*}
D_i\Big\rvert_{A(y)}(h(x))&=0, 
\end{align*}
Ainsi $D_i$ est stable sur la fibre $M_{g}(h)$. 
D'après la proposition \ref{p@}, l'intersection des strates de la
stratification $(M_{g, i})_{0 \leqslant i \leqslant g}$ de $M_{g}$ avec $M_{g}(h)$ définit la stratification
$(M_{g, i}(h))_{g-\Upsilon(h) \leqslant i \leqslant g}$ de $M_{g}(h)$. 
\end{proof}
\textbf{Remarque. }
Si $h\in \mathbb{C}^1_{2g+1}[x] $ est tel que $\Upsilon(h)=0$, d'après la proposition \ref{p39} la stratification se réduit à une
seule strate $M_{g, g}(h)$, et les $g$ champs de vecteurs $D_0, \dots, D_{g-1}$ sont linéairement indépendants aux points
$A(x)\in M_{g, g}(h)=M_g(h)$. 

\bigskip

\subsection{}\label{strati3}Nous allons établir une stratification plus fine des fibres $M_g(h)$ où $h\in
\mathbb{C}^1_{2g+1}[x]$. 
\\
Soit $h\in \mathbb{C}^1_{2g+1}[x]$, on a vu que si $\Upsilon(h)=0$, nous obtenons une stratification avec
une seule strate qui est toute la fibre. La stratification plus fine que nous allons définir coincide dans le cas
$\Upsilon(h)=0$ avec la stratification originale; nous allons exclure ce cas dans ce qui suit et supposer que $h\in \mathbb{C}^1_{2g+1}[x]$ avec $\Upsilon(h)\geqslant 1$. 

\medskip
\begin{defi*}
 Soit $Q\in \mathbb{C}[x]_{h}$ un polynôme unitaire dont le carré divise $h$. On note par $M_{g, Q}(h)$ l'ensemble des matrices 
$A(x)= \left(\begin{array}{cc}
v(x) & u(x) \\ 
w(x)& -v(x)
\end{array} \right)\in M_{g}(h)$ tel que $Q(x)$ divise 
$\PGCD(A)$. 
\end{defi*}
Une matrice $A(x)= \left(\begin{array}{cc}
v(x) & u(x) \\ 
w(x)& -v(x)
\end{array} \right)\in M_{g, Q}(h)$ alors $A(x)= Q(x)\left(\begin{array}{cc}
v'(x) & u'(x) \\ 
w'(x)& -v'(x)
\end{array} \right)$ avec $\left(\begin{array}{cc}
v'(x) & u'(x) \\ 
w'(x)& -v'(x)
\end{array} \right)\in M_{g-\deg (Q)}(\frac{h}{Q^2})$. 

\begin{prop}\label{p}
Soit $Q(x)$ un polynôme de $\mathbb{C}[x]_{h}$. Le sous-ensemble $M_{g, Q}(h)$ de $ M_{g}(h)$ est un fermé de Zariski
non-vide de $M_{g}$. 
\end{prop}
\begin{proof}[Preuve]
Soit $\deg (Q)=i$, observons que 
\begin{align}\label{etc}
 M_{g, Q}(h)= \mu_{Q}(M_{g-i}(\frac{h}{Q^{2}})). 
\end{align}
Étant donné que $M_{g-i}(\frac{h}{Q^{2}})$ est un fermé de Zariski de $M_{g}$, son image par l'application multiplication 
$\mu_{Q}$ est un fermé de Zariski de $M_{g}$. 
\\
Vérifions que le fermé $M_{g, Q}(h)\neq \emptyset$. L'application $\mathbf{H}$ est surjective, alors $ M_{g-i}(\frac{h}{Q^{2}})$ la fibre au dessus de $\frac{h}{Q^{2}}$ est non-vide, entraine que son image par $ \mu_{Q}$ est non vide, par l'egalité (\ref{etc}) $M_{g, Q}(h)\neq \emptyset $. 
\end{proof}

\goodbreak

\begin{prop}
La strate $M_{g, g} (h)$ est un ouvert de Zariski de $M_g(h)$ dont le bord topologique est $\bigcup\limits_{Q\in
\mathbb{C}[x]_{1, h}}M_{g, Q}(h)$. 
\end{prop}
\begin{proof}[Preuve]

La proposition \ref{p@} nous assure que la strate $$M_{g, g} (h)=\left\{A \in M_g(h) \mid \rho(A(x))=0\right\} \neq \emptyset. $$
Le complémentaire de la strate 
$M_{g, g} (h)$ dans
$M_{g}(h)$ est constitué des matrices $A(x)$ telles que $\rho(A(x))>0$, c'est-à-dire des matrices de
$\bigcup\limits_{Q\in \mathbb{C}[x]_{1, h}}M_{g, Q}(h)$. L'ensemble $\bigcup\limits_{Q\in
 \mathbb{C}[x]_{1, h}}M_{g, Q}(h)$ est un fermé de Zariski car par la proposition \ref{p} il est l'union finie de fermés 
 de $M_{g}(h)$. Ainsi, $M_{g, g}(h)$ est un ouvert de Zariski non-vide de $M_{g}(h)$ de bord $\bigcup\limits_{Q\in
 \mathbb{C}[x]_{1, h}}M_{g, Q}(h)$. 
\end{proof}

\begin{theo}\label{t43}
Soit $0 \leqslant i \leqslant \Upsilon(h) $. 
La strate $M_{g, g-i}(h)$ est l'union disjointe suivante: 
\begin{equation}\label{u*}M_{g, g-i}(h)=\bigsqcup_{Q\in \mathbb{C}[x]_{i, h}}\mu_{Q}M_{g-i, g-i}\left(\frac{h}{Q^{2}}\right). \end{equation}
\end{theo}

\begin{proof}[Preuve]
Soit $0 \leqslant i \leqslant \Upsilon(h)$. 
Pour tout $A(x), B(x)\in M_{g, g-i}(h)$ 
dont $ \PGCD A(x) \neq \PGCD B(x) $. Les entrées de $A(x)$ respectivement $B(x)$ ont un diviseur commun unique $Q(x)= \PGCD A(x)$ (respectivement $Q'(x) = \PGCD B(x)$) de degré $i$. Ainsi le carré du polynôme $Q(x)$ (respectivement $Q'(x)$) divise le déterminant $-h(x)$; par conséquent $A(x)$ appartient uniquement à $ \mu_{Q}M_{g-i, g-i}\left(\frac{h}{Q^{2}}\right)$ tandis que $B(x)$ appartient uniquement à $ \mu_{Q'}M_{g-i, g-i}\left(\frac{h}{Q'^{2}}\right)$ alors 
 $ \mu_{Q}M_{g-i, g-i}\left(\frac{h}{Q^{2}}\right) \cap \mu_{Q'}M_{g-i, g-i}\left(\frac{h}{Q'^{2}}\right)= \emptyset $, d'où l'union disjointe 
$$M_{g, g-i}(h)=\displaystyle{\bigsqcup\limits_{Q\in \mathbb{C}[x]_{i, h}}\mu_{Q }M_{g-i, g-i}\left(\frac{h }{Q^{2}}\right)}. $$
\end{proof}



\begin{defi*}\upshape
 Soit $1 \leqslant i \leqslant \Upsilon(h)$ et soit $Q(x)$ un polynôme de $\mathbb{C}[x]_{i, h}$. Pour tout $k \leqslant g-i$, on note 
 $$M_{g, k, Q}(h)=M_{g, Q}(h)\cap M_{g, k}(h). $$
Notez que si $ k > g-i$ alors $M_{g, Q}(h)\cap M_{g, k}(h)=\emptyset$. 
\end{defi*}

Soit $1 \leqslant i \leqslant \Upsilon(h)$ et soit $Q_{1}$ un polynôme de $\mathbb{C}[x]_{i, h}$. L'application $\mu_{Q_{1}}$ est définie comme il suit:
\begin{align}\label{53}
\begin{array}{cccl}
\mu_{Q_{1}}:& M_{g-i, g-i}(\frac{h}{Q_{1}^{2}}) &
{\longrightarrow} & M_{g, g-i, Q_{1}}(h)\\
 &A(x) & {\longrightarrow}& \mu_{Q_{1}}(A(x)). 
\end{array}
\end{align}
On note que 
$ M_{g, g-i, Q_{1}}(h)=M_{g, Q_{1}}(h)\cap M_{g, g-i}(h)$, et
en vertu de théorème \ref{t43} on a $M_{g, g-i}(h)=\displaystyle{\bigsqcup\limits_{Q\in \mathbb{C}[x]_{i, h}}\mu_{Q }M_{g-i, g-i}\left(\frac{h }{Q^{2}}\right)}$. Par conséquent, 
 $$M_{g, g-i, Q_{1}}(h)=M_{g, Q_{1}}(h)\cap \displaystyle{\bigsqcup\limits_{Q\in \mathbb{C}[x]_{i, h}}\mu_{Q }M_{g-i, g-i}\left(\frac{h }{Q^{2}}\right)}=\mu_{Q_{1} }M_{g-i, g-i}\left(\frac{h }{Q_{1}^{2}}\right). 
$$
On sait que $\mu_{Q_{1}}$ est un isomorphisme sur son image 
, alors $M_{g, g-i, Q_{1}}(h)$ est isomorphe à $M_{g-i, g-i}\left(\frac{h }{Q_{1}^{2}}\right)$. 

Nous allons introduire quelques notations qui seront utiles pour décrire une stratification plus fine de
$M_{g}(h)$. \\
Pour tout $1 \leqslant i \leqslant \Upsilon(h)$ on a par définition, $\mathbb{C}_{i, h}[x]$ est la famille finie de diviseurs quadratiques de degré $i$ de $h$. Cette famille sera notée $\displaystyle{\{Q_{j}^{(i)}(x)\}_{ 1 \leqslant j \leqslant n_i}}$. 
S'il n'y a pas de confusion et pour alléger les notations, on notera $\displaystyle{M_{g, g-i, Q_{j}^{(i)}}(h)}$ par $\displaystyle{M_{g;Q_{j}^{(i)}}(h)}$. 
\begin{defi*}\upshape
L'ensemble $\mathbb{C}[x]_{h}$ admet deux relations d'ordre $ \leqslant $ et $<$ définies de la manière suivante: pour tout polynôme unitaire $Q(x)$ et $P(x)$ de $\mathbb{C}[x]_{h}$
$$\begin{array}{cl} Q(x) \geqslant P (x)&\text{ si } Q(x) \text{ divise } P(x). \\
Q(x)>P(x) &\text{ si } \deg(Q(x))< \deg(P(x)) \text{ et } Q(x) \text{ divise } P(x) \end{array}$$ 
 \end{defi*}

\begin{theo}\label{n}
 La famille $\displaystyle{\left[ (M_{g, g-i, Q_{j}^{(i)}}(h))_{Q_{j}^{(i)}\in\mathbb{C}[x]_{h} }\right]_{1 \leqslant i \leqslant \Upsilon(h)}} $ est une stratification
 de $M_{g}(h)$. 
\end{theo}
\begin{proof}[Preuve]
Montrons que $\displaystyle{\left[ (M_{g, g-i, Q_{j}^{(i)}}(h))_{Q_{j}^{(i)}\in\mathbb{C}[x]_{h} }\right]_{1 \leqslant i \leqslant \Upsilon(h)}}$ est une partition de $M_{g}(h)$. 
Une matrice $A(x)$ de $M_g(h)$, admet un unique $\PGCD$, par conséquent elle ne peut appartenir qu'à un unique sous ensemble $M_{g, g-\rho(A), \PGCD(A)}(h)$. On conclut que les sous ensembles $\displaystyle{(M_{g, g-i, Q_{j}^{(i)}}(h))_{Q_{j}^{(i)}\in\mathbb{C}[x]_{h}}}$ sont disjoints. \\
L'isomorphisme (\ref{53}) implique que $\displaystyle{M_{g, g-i, Q_{j}^{(i)}}(h)\simeq \mu_{Q_{j}^{(i)}}M_{g-i, g-i, }\left(\frac{h}{Q_{j}^{(i)2}}\right)} $, et en vertu du théorème \ref{t43} on a:
\begin{align}\label{ff}
 M_{g, g-i}(h)=\bigsqcup_{Q_{j}^{(i)}\in \mathbb{C}[x]_{i, h}}M_{g, g-i, Q_{j}^{(i)}}(h). 
\end{align}
La proposition \ref{p39} nous affirme que
\begin{align}\label{fff}
M_{g}(h) = \bigsqcup\limits_{ i =0}^{g-1} M_{g, i}(h). 
\end{align}
En combinant les deux égalités (\ref{ff}) et (\ref{fff}), on obtient
\begin{align*}
M_{g}(h) &= \displaystyle{ \bigsqcup\limits_{ i =0}^{g-1} \bigsqcup\limits_{Q_{j}^{(i)}\in \mathbb{C}[x]_{i, h}}M_{g, g-i, Q_{j}^{(i)}}(h)}= \displaystyle{\bigsqcup\limits_{Q_{j}^{(i)}\in \mathbb{C}[x]_{h}}M_{g, g-i, Q_{j}^{(i)}}(h)}. 
\end{align*}
On conclue que la famille $\displaystyle{(M_{g, g-i, Q_{j}^{(i)}}(h))_{Q_{j}^{(i)}\in \mathbb{C}[x]_{h} }} $ est bien une partition de $M_{g}(h)$. 

\bigskip

Montrons maintenant que $\overline{M_{g, g-i, Q_{j}^{(i)}}(h)}=\displaystyle{ \bigsqcup\limits_{ Q_{j'}^{(i')} \leqslant Q_{j}^{(i)}}M_{g, g-i', Q_{j'}^{(i')}}(h)}$. 
Commençons par montrer l'égalité suivante: 
\begin{align}\label{las}
 M_{g, Q_{j}^{(i)}}(h)=\displaystyle{ \bigsqcup\limits_{ Q_{j'}^{(i')} \leqslant Q_{j}^{(i)}}M_{g, g-i', Q_{j'}^{(i')}}(h)}. 
 \end{align}
Par définition, $M_{g, Q_{j}^{(i)}}$ est l'ensemble des matrices $A(x)\in M_{g}(h)$ telles que $ Q_{j}^{(i)}(x)$ divise
$\PGCD(A(x))$. Aussi par définition l'ensemble $M_{g, g-i', Q_{j'}^{(i')}}(h)$ est constitué des matrices $A(x)\in M_{g}(h)$
telles que $\PGCD(A(x))=Q_{j'}^{(i')}(x)$. Comme $Q_{j'}^{(i')}(x) \leqslant Q_{j}^{(i)}(x)$, le polynôme $Q_{j}^{(i)}(x)$ divise $Q_{j'}^{(i')}(x), $ on obtient alors $$\displaystyle{
 \bigsqcup\limits_{ Q_{j'}^{(i')} \leqslant Q_{j}^{(i)}}M_{g, g-i', Q_{j'}^{(i')}}(h)}= M_{g, Q_{j}^{(i)}}(h). $$ 
 Le complémentaire de $M_{g, g-i, Q_{j}^{(i)}}(h)$ dans $M_{g, Q_{j}^{(i)}}(h)$ est
 $$\displaystyle{ \bigsqcup\limits_{ Q_{j'}^{(i')}< Q_{j}^{(i)}}M_{g, g-i', Q_{j'}^{(i')}}(h)}= \displaystyle {\bigsqcup\limits_{i'> i}M_{g, g-i'} \cap M_{g, Q_{j}^{(i)}}(h)}
. $$ 
On sait que
$\bigsqcup\limits_{i'> i}M_{g, g-i'}$ est un fermé de Zariski, alors l'intersection $\displaystyle {\bigsqcup\limits_{i'> i}M_{g, g-i'} \cap M_{g, Q_{j}^{(i)}}(h)}$ est un
fermés de Zariski de $M_{g, Q_{j}^{(i)}}(h)$, et l'ensemble ${M_{g, g-i, Q_{j}^{(i)}}(h)}$ est un ouvert de Zariski dans un fermé de Zariski.
 Donc $$\overline{M_{g, g-i, Q_{j}^{(i)}}(h)}=\displaystyle{ \bigsqcup\limits_{
 Q_{j'}^{(i')} \leqslant Q_{j}^{(i)}}M_{g, Q_{j'}^{(i')}}(h)}. $$ On conclut que la famille
$\displaystyle{(M_{g, g-i, Q_{j}^{(i)}}(h))_{Q_{j}^{(i)} \in\mathbb{C}[x]_{h} }} $ est une stratification de
$M_{g}(h)$. 
\end{proof}

\bigskip

\begin{defi*}
Les strates $M_{g, g-i, Q_{j}^{(i)}}(h)$ de la stratification $\displaystyle{\left[ (M_{g, g-i, Q_{j}^{(i)}}(h))_{Q_{j}^{(i)}\in\mathbb{C}[x]_{h} }\right]_{1 \leqslant i \leqslant \Upsilon(h)}} $ 
 de $M_{g}(h)$. sont appelée les strates \emph{fines} 
 du système de
 Mumford d'ordre $g$. 
\end{defi*}

\bigskip

Chaque fibre $M_{g}(h)$ du système de Mumford d'ordre $g$ admet donc la stratification définie ci-dessous 
$$\displaystyle{\left(M_{g, g-i, Q_{j}^{(i)}}(h)\right)_{Q_{j}^{(i)}\in\mathbb{C}[x]_{h} }}, $$ où les strates 
 sont caractérisées par le degré de liberté des champs de vecteurs $(D_{i})_{ 0 \leqslant i \leqslant g-1}$ qui est la dimension des espaces vectoriels 
 $\langle D_{i} \, | \, 0 \leqslant i \leqslant g-1 \rangle$ et par les
diviseurs quadratiques de $h$.

\bigskip

\subsection{}\label{lissitude} Lissitude des strates. 
Dans ce paragraphe, nous allons déterminer les singularités de chaque fibre $M_{g}(h)$, ainsi que la fermeture de
chaque strate $M_{g, i}(h)$ de $M_{g}(h)$. Pour ce faire nous allons déterminer, en tout point le rang de la matrice
jacobienne de l'application moment $ \mathbf{H}$. Nous pourrons conclure que la dimension de chaque fibre $M_g(h)$
est égale à $g$. 

\medskip

Chaque fibre $M_{g}(h)$ de l'espace de phase $M_{g}$ est de dimension au moins $g$ et est munie de $g$ champs de
vecteurs $(D_{i})_{ 0 \leqslant i \leqslant g-1}$. Elle admet une seule strate $M_{g, g}$ où les champs de vecteurs sont linéairement
indépendants. qu'on appelle $M_{g, g}(h)$ la strate \emph{maximale} de $M_{g}(h)$.

\smallskip

Rappelons qu'on note par $ \mathbf{H}$ l'application polynomiale surjective définie par (\ref{H}); 
où pour tout $ A(x)=\left(\begin{array}{cc} v(x) & u(x) \\w(x)&
 -v(x)\end{array}\right)\in M_{g} $ on a $\mathbf{H}(A(x))= -\det(A(x))= x^{2g+1}+\sum\limits_{i=0}^{2g}h_i(A(x))x^i $, 
est un polynôme en $x$ de degré $2g+1$;
ses coefficients $(h_{i})_{ 0 \leqslant i \leqslant 2g}$ sont les fonctions polynomiales en fonction des fonctions coordonnées $\left[ (u_{i})_{ 0 \leqslant i \leqslant g-1}, \right. $ $
(v_{i})_{ 0 \leqslant i \leqslant g-1}, $ $\left. (w_{i})_{ 0 \leqslant i \leqslant g}\right]$ de $M_g$. \\
Pour tout polynôme $h(x)\in \mathbb{C}_{2g+1}^1[x]$. La fibre $M_g(h)$ au dessus de $h(x)$ par $\mathbf{H}$ est de dimension \begin{align*}M_g(h)&= \dim M_g- \dim (\im(\mathbf{H})), \\
&=3g+1-(2g+1), \\
M_g(h)&=g. 
\end{align*}

\bigskip

Rappelons la définition de la matrice jacobienne 
\begin{defi*}
Soit $\mathbf{F}=(F_l)_{1 \leqslant l \leqslant k}$ une fonction vectorielle de $V= \langle t_{i} \, \rvert \, 1 \leqslant i \leqslant n \rangle$ dans $\mathbb{C}^k$. 
La matrice jacobienne de $\mathbf F$ est la matrice des dérivées partielles du premier ordre d'une fonction vectorielle en un point donné $a\in V$, notée $J_\mathbf{F}\rvert_{a}$ avec les entrées 
$${j_{l, i}\rvert_{a}= \left(\frac{\partial F_l}{\partial t_i}\biggr\rvert_{a} \right), \text{ pour }
1 \leqslant l \leqslant k, \text{ et } 1 \leqslant i \leqslant n} $$
\end{defi*}
Supposons que $k\leqslant n$. 
Un point $a$ de $V$ est dit lisse de dimension $n-k$, si la matrice jacobienne au point $a$, $J_\mathbf{F}\rvert_{a}$ est de rang maximal $k$. 
Il nous arrivera d'appeler la matrice jacobienne de $\mathbf{F}$ par la jacobienne de $\mathbf{F}$.

\bigskip

Rappelons que $M_g$ est la variété algébrique affine donnée par les coefficients de $u(x), v(x), w(x)$ comme vecteurs de coordonnées, donc isomorphes à $C^{3g+1}$ en particulier irréductible. 
\\
Soit l'application surjective $\mathbf{H}$ définie par (\ref{H}), $\mathbf{H}(A)=uw+v^2$ qui est un polynôme unitaire en $ x$ de degré $2g+1$ ses coefficients non triviaux $(h_0, h_1, ... h_{2g})$ sont des polynômes en $(u_j)_{0 \leqslant j \leqslant g-1 }, $ $ (v_j)_{0 \leqslant j \leqslant g-1 } $, $ (w_i)_{0 \leqslant j \leqslant g}$. 
Soit $h(x)=x^{2g+1}+\sum\limits_{i=0}^{2g} a_ix^i $. On considère la sous-variété fermée $M_g(h)$ donnée par les $2g+1$ fonctions régulières $(h_i)_{0 \leqslant i \leqslant 2g }$ sur $M_g$ telles que $h_i-a_i=0$ pour $0 \leqslant j \leqslant g-1$. 
L'application $\mathbf{H}$ est surjective, alors la dimension de toute fibre $M_g(h)$ est $3g+1-(2g+1)=g$. 
Pour éviter toute confusion, nous notons par $u\rvert_{A^0}, v\rvert_{A^0}, w \rvert_{A^0}$ les polynômes tels que les coefficients des polynômes $u, v, w$ sont évalués au point $A^0 \in M_g$. 
\\
D'après Shafarevitch \cite{S}, la dimension de l'espace tangent en tout point $
A^0\in M_g(h)$ est donnée par $\dim M_g(h) – \rk J_{\mathbf{H}}\rvert_{A^0}$. 
%
 
La jacobienne $J_{\mathbf{H}}\rvert_{A^0}$ est \begin{equation}\label{b1}
 \frac{\partial \mathbf{H}}{\partial \tau} \biggr\rvert_{A^0} = \sum\limits_{i=0}^{2g}x^{i}\frac{\partial h_{i}}{\partial
 \tau}\biggr\rvert_{A^0}\;, 
\end{equation}
d'après $\mathbf{H}(A)=uw+v^2$ on a 
\begin{equation} \label{b2}
 \frac{\partial \mathbf{H}}{\partial \tau} \biggr\rvert_{A^0} =2v(x)\frac{\partial v(x)}{\partial \tau}\biggr\rvert_{A^0}+
 w(x)\frac{\partial u(x)}{\partial \tau}\biggr\rvert_{A^0}+u(x)\frac{\partial w(x)}{\partial \tau}\biggr\rvert_{A^0}\;. 
\end{equation}
Si $\tau$ est une des fonctions 
coefficients coordonnées $\{u_j\}_{0 \leqslant j \leqslant g-1 } \cup \{v_j\}_{0 \leqslant j \leqslant g-1 }\cup \{w_j\}_{0 \leqslant j \leqslant g }$
alors les égalités (\ref{b1}) et
(\ref{b2}) deviennent
\begin{equation}\label{b1*}
 \frac{\partial \mathbf{H}}{\partial u_{j}} \biggr\rvert_{A^0} = \sum\limits_{i=0}^{2g}x^{i}\frac{\partial h_{i}}{\partial
 u_{j}}\biggr\rvert_{A^0}, 
 \quad\frac{\partial \mathbf{H}}{\partial v_{j}} \biggr\rvert_{A^0} = \sum\limits_{i=0}^{2g}x^{i}\frac{\partial
 h_{i}}{\partial v_{j}}\biggr\rvert_{A^0}, 
 \quad\frac{\partial \mathbf{H}}{\partial w_{j}} \biggr\rvert_{A^0} =
 \sum\limits_{i=0}^{2g}x^{i}\frac{\partial h_{i}}{\partial w_{j}}\biggr\rvert_{A^0}\;, 
\end{equation}
et
\begin{equation}\label{b2*}
 \frac{\partial \mathbf{H}}{\partial u_{j}} \biggr\rvert_{A^0} =x^{j}w^{0}(x), 
 \quad\frac{\partial \mathbf{H}}{\partial v_{j}}
 \biggr\rvert_{A^0} =2x^{j}v^{0}(x), 
 \quad\frac{\partial \mathbf{H}}{\partial w_{j}} \biggr\rvert_{A^0} =x^{j}u^{0}(x). 
\end{equation}
En combinant les égalités de (\ref{b1*}) avec ceux de (\ref{b2*}), on obtient
\begin{equation}\label{7. }
 \sum\limits_{i=0}^{2g}x^{i}\frac{\partial h_{i}}{\partial u_{j}} \biggr\rvert_{A^0}= x^{j}w^{0}(x), \quad
 \sum\limits_{i=0}^{2g}x^{i}\frac{\partial h_{i}}{\partial v_{j}}
 \biggr\rvert_{A^0}=2x^{j}v^{0}(x), 
 \quad\sum\limits_{i=0}^{2g}x^{i}\frac{\partial h_{i}}{\partial w_{j}} \biggr\rvert_{A^0}=
 x^{j}u^{0}(x). 
\end{equation}

On a 
$J_\mathbf{H}\rvert_{A^0}$ avec les entrées 
\begin{align}\label{jac}
j_{i, j}=\left\{ \begin{array}{ll} w_{i-j}\rvert_{A^0} & \text{ pour } 1 \leqslant i \leqslant g, \\
2v_{i-j-g}\rvert_{A^0} & \text{ pour } g+1 \leqslant i \leqslant 2g, \\
u_{i-j-2g}\rvert_{A^0} & \text{ pour } 2g+1 \leqslant i \leqslant 3g+1, 
\end{array}\right. 
\end{align}
avec la convention suivante si $j \notin [0, g], k \notin [0, g-1]$ et $ l \notin [0, g+1]$ on a $u_j= 0, v_k= 0 $ et $z_l= 0 $. \\
On note par $(0^{i\times j})$ la matrice nulle de dimension $i\times j$. \\
Par la définition de la matrice Toeplitz (\ref{Matrx}) on peut réécrire $J_\mathbf{H}\biggr\rvert_{A^0}$ comme il suit 
\begin{align}\label{jaco}J_\mathbf{H}\biggr\rvert_{A^0}=\left(\begin{array}{l} M_{w, g}^t\rvert_{A^0} \\
M_{v, g}^t\rvert_{A^0} 0^{g\times 1} \\
M_{u, g+1}^t\rvert_{A^0} 
\end{array}\right), 
\end{align}
les matrices $M_{w, g}^t\rvert_{A^0}, 
M_{v, g}^t\rvert_{A^0}, 
M_{u, g+1}^t\rvert_{A^0}$ sont les matrices transposées de $M_{w, g}\rvert_{A^0}, 
M_{v, g}\rvert_{A^0}, 
M_{u, g+1}\rvert_{A^0}$ respectivement. 
\bigskip

%
La dimension de la matrice $J_\mathbf{H}
$ est $(3g+1)\times(2g+1)$ définissant une application linéaire de $M_g(h)$ vers $ \mathbb{C}^{2g+1}$
. \\


\begin{theo}\label{PG}
Soit $A^{0}(x)=\left(\begin{array}{cc} v^{0}(x) & u^{0}(x) \\w^{0}(x)& -v^{0}(x)\end{array}\right)\in M_{g}$. 
Le rang de la matrice jacobienne de $\mathbf{H}$ au point $A^{0}$ est égal à $2g+1-\rho(A^{0}(x))$. 
\end{theo}
\begin{proof}[Preuve]
De la définition (\ref{jaco}), la jacobienne $J_\mathbf{H}\rvert{A^{0}}$ au point $A^0$ et de (\ref{noyau}) on a 
 $$\dim \ker J_\mathbf{H}\rvert{A^{0}} = \deg \PGCD (u^0, v^0, w^0)=\rho(A^{0});$$
 par conséquent le rang de la matrice $J_\mathbf{H}\rvert{A^{0}}$ est
 $$\rk J_\mathbf{H}\rvert{A^{0}} = 2g+1-\rho(A^{0}(x)).$$ 

 \end{proof}
 \bigskip
Après \cite[Chapter II, Section 1.3]{S}, la dimension de l'espace tangent en un point $A^0(x)\in M_g(h)$ est donnée par $\dim M_g(h)-$ $\dim(\ker(J_\mathbf{H}(A^{0}(x)))$. Nous allons utiliser ce fait pour prouver la proposition suivante: %

 \begin{prop}\label{SM}
 Soit $h\in \mathbb{C}[x]^1_{2g+1}$. La strate maximale $M_{g, g}(h)$ est une variété quasi-affine lisse de dimension $g$, le bord de $M_{g, g}(h)$ est constitué de tous les points singuliers de $M_{g}(h)$. 
 \end{prop}
 \begin{proof}[Preuve]
Une matrice $A^{0}(x)\in M_g(h)$ appartient à la strate maximale $M_{g, g}(h)$ si et seulement si $\rho(A^{0}(x))=0$. 
Par définition on a qu'un point $A^0(x)$ de $M_g(h)$ est dit lisse en $M_g(h)$ 
si 
le rang de $J_{\mathbf{H}}(A^0(x))$ est maximal $2g+1$ 
c'est à dire $\rho(A^{0}(x))=0$. Le théorème \ref{PG} nous assure que $A^0(x)$ est dans la partie lisse de $M_{g}(h)$ si et seulement si $A^0(x)\in M_{g, g}(h)$ 
. La strate maximale $M_{g, g}(h)$ de $M_g(h)$ est composée de tous les points lisses de $M_g(h)$. Par conséquent le complémentaire $M_{g, g}(h)$ dans $M_g(h)$ est composé de tous les points singuliers et il est égal à $\bigsqcup\limits_{j \leqslant g-1 } M_{g, j}$. 
\end{proof}

De la proposition \ref{SM} et du fait que chaque strate $M_{g, i}(h)$ est l'union disjointe de sous-variétés
quasi-affines isomorphes à des strates maximales de systèmes de Mumford d'ordre $i$, on peut déduire directement le
corollaire suivant:
\begin{coro}\label{cor3.3}
 Soit $h\in \mathbb{C}^{2g+1}_1[x]$ et soit $i<g$. Chaque strate $M_{g, i}(h)$ de la fibre $M_{g}(h)$ est lisse et de dimension $i$, composée de sous-variétés 
 équidimensionelle de dimension $ i $ de $M_{g, i}(h)$. 
 \end{coro}
 \begin{proof}[Preuve]
 Rappelons que $
 \mathbb{C}_{g-i, h}[x] 
 $ est la famille finie de diviseurs quadratiques de degré $g-i$ de $h$. Soit $Q\in \mathbb{C}_{g-i, h}[x]$. La proposition \ref{SM} implique que l'unique strate maximale $M_{i, i}(\frac{h}{Q^2})$ de $M_{i}(\frac{h}{Q^2})$, qui est lisse de dimension $i$. L'image de $M_{i, i}(\frac{h}{Q^2})$ par l'isomorphisme $\mu_{Q}$ est un fermé irréductible de $M_{g, i}(h)$ de dimension $i$. On a donc 
 $$M_{g, i}(h)= \bigsqcup\limits_{ Q(x)\in\mathbb{C}_{g-i, h}[x] } \mu_{Q}M_{i, i}(\frac{h}{Q^2}). $$
La variété quasi-affine $M_{g, i}(h)$ est composée d'union disjointe de sous-variétés 
 équidimensionelle de dimension $ i $. 
 \end{proof}
\section*{Appendice}\label{ap}
Nous exposons ici la description originale de Mumford \cite{Mum} des polynômes $u(x)$ et $v(x)$. \\

On note par $\mathcal{C}^g_s$, le produit $g$ symétrique de $\mathcal{C}$. 
\\
Soit $\Delta$ un sous ensemble de $\mathcal{C}^g$, composé de $g$-uplet $\{((x_1, y_1), (x_2, y_2), \cdots, (x_g, y_g))\} \in\mathcal{C}^g_s$ tel qu'il existe $1 \leqslant i \neq j \leqslant g $ où $ y_i = - y_j\neq0$. On note par $n_i$ le nombre de fois où le couple $(x_i, y_i)$ apparaît dans le $g$-uplet. 


\bigskip

Fixons une courbe hyperelliptique $\mathcal{C}$ associée à un polynôme $h\in \mathbb{C}^1_{2g+1}[x]$. Nous allons adjoindre à chaque élément de $\mathcal{C}^g_s/\Delta$ un couple de polynômes par l'application $\phi$ suivante:
\begin{equation}\label{phi}\begin{array}{cccc}
\phi:& \mathcal{C}^g_s/\Delta& \longrightarrow & \mathbb{C}^1_g[x] \times \mathbb{C}_{g-1}[x]\\
& ((x_1, y_1), (x_2, y_2), \cdots, (x_g, y_g)) & \longrightarrow & (u(x), v(x))
\end{array}
\end{equation}
La définition du polynôme $u(x)$ est en corrélation avec les premières composantes $\{x_i \}_{ 1 \leqslant i \leqslant g}$ du $g$-uplet de $\mathcal{C}^g_s$ tel que:
 \begin{equation*}
u(x)=\prod\limits_{i=1}^g(x-x_i), \\
\end{equation*}
La description du polynôme $v(x)$ se fera à l'aide des équations linaires suivantes: \\
Pour tout $1 \leqslant i \leqslant g$, et $ 0 \leqslant k_i \leqslant {n}_i-1$, 
\begin{equation}\label{polyv}
\displaystyle{
\frac{d^k}{dt}(v(x)\pm\sqrt{h(x)})\Big\rvert _{x=x_i}=0 \;\;\; 
 \text{ si } y_i= \mp\sqrt{h(x_i)} \text{ \; et \;} \frac{d^k}{dt}u(x)\Big\rvert _{x=x_i}=0. 
 }
\end{equation}
Notez que lorsque les 
 $(n_i)_{1 \leqslant i \leqslant g}=(1)_{1 \leqslant i \leqslant g}$, c'est à dire quand $(x_i)_{1 \leqslant i \leqslant g}$ sont tous distincts, le polynôme ${v(x)=\sum\limits_{i=1}^{g}y_i\prod\limits_{j=1, j \neq i }^g {\frac{ x-x_j}{x_i-x_j}}}$.

\section*{Notations}
Section \ref{section2}: $\mathcal{C}$, $u, v, w$, 
$u_k$, $v_k$, $M_g$, $\mathbf{H}$, 
$\mathbf{H}_z$, $D_z$, 
$D_i$. \\
Section \ref{section3}: $\rho$, 
$(S_{g, i})$,$\sigma$, $ M_{g, i}$,$\mu_{P}$.\\
Section \ref{section4}: $\mathbb{C}_k[x]$, 
 $\mathbb{C}^1_k[x]$, $\mathbb{C}[x]_{h}$
,$M_{g, Q}(h)$, $M_{g, k, Q}(h)$ 
,$J_\mathbf{H}(A^{0})$.

\end{document}